\documentclass[12pt,a4paper]{article}
\usepackage[utf8]{inputenc}
\usepackage{amsmath}
\usepackage{breqn}
\usepackage{amsthm}
\usepackage{amsfonts}
\usepackage{amssymb} 
\usepackage{graphicx}
\usepackage[left=2cm,right=2cm,top=2cm,bottom=2cm]{geometry}
\usepackage[numbers]{natbib}
\usepackage{lineno,hyperref}
\modulolinenumbers[5]

\newtheorem{theorem}{Theorem}[section]
\newtheorem{corollary}{Corollary}[section]
\newtheorem{lemma}{Lemma}[section]

\newtheorem*{keyword}{Keywords}
\newtheorem*{remark}{Remark}
\newtheorem{example}{Example}[section]
\newtheorem*{definition}{Definition}

\newcommand*{\Scale}[2][4]{\scalebox{#1}{$#2$}}

\begin{document}
\title{Multiple Sums and Partition Identities}
\author{Roudy El Haddad \\
Universit\'e La Sagesse, Facult\'e de g\'enie, Polytech \\
\href{roudy1581999@live.com}{roudy1581999@live.com}}
\date{}
\maketitle

\begin{abstract}
Sums of the form $\sum_{q \leq N_1 < \cdots < N_m \leq n}{a_{(m);N_m}\cdots a_{(2);N_2}a_{(1);N_1}}$ date back to the sixteen century when Vi\`ete illustrated that the relation linking the roots and coefficients of a polynomial had this form. In more recent years, such sums have become increasingly used with a diversity of applications. In this paper, we develop formulae to help with manipulating such sums (which we will refer to as \textit{multiple sums}). We develop variation formulae that express the variation of multiple sums in terms of lower order multiple sums. Additionally, we derive a set of \textit{partition identities} that we use to prove a reduction theorem that expresses multiple sums as a combination of simple (non-recurrent) sums. We present a variety of applications including applications concerning polynomials and MZVs as well as a generalization of the binomial theorem. Finally, we establish the connection between multiple sums and a type of sums called recurrent sums. By exploiting this connection, we provide additional \textit{partition identities} for odd and even partitions. 
\end{abstract}
\begin{keyword}
{\em Multiple sums, Vi\`ete's Formula, Vi\`ete's Theorem, Polynomials, Partitions, Stirling numbers of the first kind, Multiple zeta values, Riemann zeta function, Bernoulli numbers, Faulhaber formula. } \\
\thanks{\bf{MSC 2020: }{primary 11P84 secondary 11B73, 11M32, 11C08}}
\end{keyword}
\section{Introduction and Notation} \label{intro}
The harmonic series is a divergent series as was proven independently by Nicole Oresme \cite{Oresme}, Pietro Mengoli \cite{Mengoli}, Johann Bernoulli \cite{JohannBernoulli}, and Jacob Bernoulli \cite{JacobBernoulli1,JacobBernoulli2}. 
However, elevating the terms to a power $s>1$, we obtain a convergent series. Sums of the form 
$$
\zeta(s)=\sum_{n=1}^{\infty}{\frac{1}{n^s}}
$$ 
where $s$ is a real number, were first studied by Euler. The case for $s=2$ is the celebrated Basel problem. Its fame came from the originality of the way Euler proved that $\zeta(2)=\frac{\pi^2}{6}$ \cite{Basel1,Basel2,Basel3}. More rigorous proofs were later developed \cite{Basel4}. By the method used to find $\zeta(2)$, Euler was able to obtain a general formula for this zeta function for positive even values of $s$. 
In 1859, a generalized form of $\zeta(s)$ where $s$ is a complex variable was introduced by Riemann in his article ``On the Number of Primes Less Than a Given Magnitude'' \cite{riemann1974appendix}. In recent years, mathematicians like Zagier \cite{zagier1995multiple}, Hoffman \cite{hoffman1992multiple}, and Granville \cite{granville1997decomposition} have introduced a generalized form of the zeta function which they call the multiple harmonic series (MHS) or multiple zeta values (MZV). However, interest in these sums dates back to when Euler studied the case of length 2 \cite{euler}. 

These sums/series have been of interest to mathematicians for a long time. Since the beginning of the 1990s, such series/sums have been heavily studied by mathematicians such as Hoffman and Zagier. Interest in such sums grew from their tremendous importance in Number Theory and numerous applications \cite{zagier1994values}. For example, the multiple harmonic series is directly related to the Riemann zeta function $\zeta(s)$ \cite{hoffman1992multiple,granville1997decomposition}. Interest in these series extends beyond mathematics. In fact, they appear in many fields of physics. The number $\zeta(\overline{6},\overline{2})$ appeared in the quantum field theory literature in 1986 \cite{broadhurst1986exploiting}. They are of fundamental importance for the connection of knot theory with quantum field theory \cite{broadhurst2013multiple,kassel2012quantum}. Multiple harmonic sums became even more important when higher order calculations in quantum electrodynamics (QED) and quantum chromodynamics (QCD) started needing them \cite{blumlein1999harmonic,blumlein2010multiple}. 

A multiple harmonic series (MHS) or multiple zeta values (MZV) is defined as: 
$$
\zeta(s_1,s_2, \ldots, s_k)
=\sum_{1 \leq N_1 < N_2 < \cdots < N_k}{\frac{1}{N_1^{s_1}N_2^{s_2}\cdots N_k^{s_k}}}
.$$ 
Its partial sum is written as: 
$$
\zeta_{n}(s_1,s_2, \ldots, s_k)
=\sum_{1 \leq N_1 < N_2 < \cdots < N_k \leq n}{\frac{1}{N_1^{s_1}N_2^{s_2}\cdots N_k^{s_k}}}
.$$ 
These sums are a particular case of what we called multiple sums as they are of the form $\sum_{1\leq N_1 < \cdots < N_m \leq n}{a_{(m);N_m} \cdots a_{(1);N_1}}$ with $a_{(i);N_i}=\frac{1}{N_i^{s_i}}$ for all $i$. Note, however, that while this particular case has been heavily studied, the general case has received much less attention. Meanwhile, there are thousands of formulae for multiple harmonic sums/series in the literature, not much can be found for general multiple sums/series. In this article, we derive some essential formulae for dealing with general multiple sums. 
The aim of this paper is to develop methods/formulae that allow us to better the way we work with multiple sums. We will develop formulae to calculate the variation of such sums as well as formulae to express multiple sums in terms of simple sums. Partition identities needed to prove these formulae as well as partition identities that can be derived from these formulae will be presented. These identities include sums over partitions involving Bernoulli numbers and the zeta function as well as a new definition of binomial coefficients as a sum over partitions. Likewise, by complementing this work with that done in \cite{RecurrentSums}, we will extend the partition identities of this article to sums over odd and even partitions. As a matter of fact, multiple sums are intimately related to recurrent sums (presented in \cite{RecurrentSums}) as we will show in this paper. 

Additionally, note that like the MHS, the general multiple sum structure is not recent, it goes back to the 17-th century. In 1646, Vi\`ete proved that a polynomial can be represented as a product of factors (Vi\`ete’s theorem) as well as he developed the relations linking the coefficients of a polynomial to its roots for positive roots (Vi\`ete’s formula) \cite{Viete}. Vi\`ete’s formula was later proven to hold for any roots or coefficients by A. Girard \cite{Girard}. The type of sums represented by Vi\`ete’s formula is a general multiple sum. Hence, according to Vi\`ete's formula, multiple sums are fundamental for linking the roots and coefficients of a polynomial. 

The importance of this article is based on how it improves our ability to study sums of this form. The theorems presented in this paper can be used to develop new theorems involving multiple sums or to improve upon previously obtained results as we will illustrate in this article. Applications include generalizing the Faulhaber formula for the sum of powers to a formula for the multiple sum of powers. They also include simplifying the relation between roots and coefficients of a polynomial as well as linking the roots of a polynomial to those of its derivatives. And probably, most importantly, illustrating the behavior of MZVs as the number of sums goes to infinity. 
The partition identities are also a major part of the importance of this paper. This article includes partition identities not only for partitions in general but also for odd and even partitions. These identities could be key in deriving new theorems involving odd or even partitions.

Now let us define a notation for multiple sums which we will use in the remainder of this paper: For any $m,q,n \in \mathbb{N}$ where $n \geq q+m-1$ and for any set of sequences $a_{(1);N_1},\ldots,a_{(m);N_m}$ defined in the interval $[q,n]$, let $P_{m,q,n}(a_{(1);N_1},\ldots,a_{(m);N_m})$ represent the general multiple sum of order $m$ for the sequences $a_{(1);N_1},\ldots,a_{(m);N_m}$ with lower and upper bounds respectively $q$ and $n$. For simplicity, however, we will denote it simply as $P_{m,q,n}$.
\begin{equation} \label{eq1}
\begin{split}
P_{m,q,n}
&=\sum_{q \leq N_1 < \cdots < N_m \leq n}{a_{(m);N_m}\cdots a_{(2);N_2}a_{(1);N_1}}\\
&=\sum_{N_m=q+m-1}^{n}{\cdots \sum_{N_2=q+1}^{N_3-1}{\sum_{N_1=q}^{N_2-1}{a_{(m);N_m}\cdots a_{(2);N_2}a_{(1);N_1}}}} \\
&=\sum_{N_m=q+m-1}^{n}{a_{(m);N_m} \cdots \sum_{N_2=q+1}^{N_3-1}{a_{(2);N_2}\sum_{N_1=q}^{N_2-1}{a_{(1);N_1}}}} 
.\end{split}
\end{equation}
The most common case of a multiple sum is that where all sequences are the same, 
\begin{equation} \label{eq2}
\begin{split}
P_{m,q,n}(a_{N_1}, \ldots , a_{N_m})
&=\sum_{q \leq N_1 < \cdots < N_m \leq n}{a_{N_m}\cdots a_{N_2}a_{N_1}}\\
&=\sum_{N_m=q+m-1}^{n}{\cdots \sum_{N_2=q+1}^{N_3-1}{\sum_{N_1=q}^{N_2-1}{a_{N_m}\cdots a_{N_2}a_{N_1}}}} \\
&=\sum_{N_m=q+m-1}^{n}{a_{N_m} \cdots \sum_{N_2=q+1}^{N_3-1}{a_{N_2}\sum_{N_1=q}^{N_2-1}{a_{N_1}}}} 
.\end{split}
\end{equation}
For simplicity, we will denote it as $\hat{P}_{m,q,n}$.
\begin{remark}
{\em
Knowing that adding zeros to a sum does not change the sum and noticing that for $N_2=q$, we get $q \leq N_1 \leq N_2 - 1 =q-1$ which would lead to an empty sum for this value. Hence, we can start $N_2$ at $q$. Similarly, for $N_3=q,q+1$ or $\cdots$ or $N_{m}=q, \ldots, q+m-2$, all would lead to zeros. Hence, we could start all the variables at $q$. }
$$
P_{m,q,n}
=\sum_{N_m=q}^{n}{\cdots \sum_{N_2=q}^{N_3-1}{\sum_{N_1=q}^{N_2-1}{a_{(m);N_m}\cdots a_{(2);N_2}a_{(1);N_1}}}}
.$$
\end{remark}
\begin{remark}
{\em If $m > n-q+1$ $($or $n<q+m-1)$, the multiple sum can still be considered defined and will be zero $(P_{m,q,n}=0)$. }
\end{remark}
\begin{remark}
{\em
A multiple sum of order 0 is always equal to $1$ $(P_{0,i,j}=1, \forall i,j \in \mathbb{N})$. It is not equivalent to an empty sum (which is equal to $0$). }
\end{remark}
\begin{remark}
{\em From these expressions, we can clearly see that it is directly related to recurrent sums \cite{RecurrentSums}. }
\end{remark} 
In this paper, multiple sums will be studied. In Section \ref{Variation Formulas}, formulas for the calculation of variation of these sums in terms of lower order multiple sums will be presented. Then, in Section \ref{Reduction Formulas}, we will present a reduction formula that allows the representation of a multiple sum as a combination of simple (non-recurrent) sums. In section \ref{ApplicationstoPolynomials}, the relations developed will be applied to Vi\`ete's formula in order to simplify the relation linking the coefficients of a polynomial to its roots. Additionally, some theorems related to polynomials will be developed. A generalization of the binomial theorem will also be developed. 
In section \ref{Applications}, the reduction theorem will be used to calculate certain special sums such as the multiple harmonic sum and the multiple power sum. In section \ref{Relation to Recurrent Sums and Odd-Even Partitions}, we investigate the relation between recurrent sums and multiple sums then, using these links, we derive some odd and even partition identities.     
\section{Variation Formulas} \label{Variation Formulas}
In this section, we will develop formulas to express the variation of a multiple sum of order $m$ ($P_{m,q,n+1}-P_{m,q,n}$) in terms of lower order multiple sums. Equivalently, these formulas can be used to express $P_{m,q,n+1}$ in terms of $P_{m,q,n}$ and lower order multiple sums. 
\subsection{Simple expression}
We begin by presenting the simplest case of the variation formula in Lemma \ref{Lemma 2.1}. This basic form is needed in order to prove the general form. 
\begin{lemma} \label{Lemma 2.1}
For any $m,q,n \in \mathbb{N}$ where $n \geq q+m-1$, we have that 
$$P_{m,q,n+1}=P_{m,q,n}+a_{(m);n+1} P_{m-1,q,n}.$$
\end{lemma}
\begin{proof}
\begin{equation*} 
\begin{split}
P_{m,q,n+1}
&=\sum_{q \leq N_1 < \cdots <N_m \leq n+1}{a_{(m);N_m} \cdots a_{(1);N_1}} \\
&=\sum_{q \leq N_1 < \cdots <N_m < n+1}{a_{(m);N_m} \cdots a_{(1);N_1}}+\sum_{q \leq N_1 < \cdots N_{m-1}<N_m = n+1}{a_{(m);N_m} \cdots a_{(1);N_1}} \\
&=\sum_{q \leq N_1 < \cdots <N_m \leq n}{a_{(m);N_m} \cdots a_{(1);N_1}} +a_{(m);n+1} \sum_{q \leq N_1 < \cdots <N_{m-1} \leq n}{a_{(m-1);N_{m-1}} \cdots a_{(1);N_1}}\\
&=P_{m,q,n}+a_{(m);n+1} P_{m-1,q,n}
.\end{split}
\end{equation*}
\end{proof}
Based upon Lemma \ref{Lemma 2.1}, a more generalized version of the variation formula can be developed which allows the representation of $P_{m,q,n+1}$ in terms of $P_{m,q,n}$ and lower order multiple sums (from order $0$ to $(m-1)$). 
\begin{theorem} \label{Theorem 2.1}
For any $m,q,n \in \mathbb{N}$ where $n \geq q+m-1$ and for any set of sequences $a_{(1);N_1},\ldots,a_{(m);N_m}$ defined in the interval $[q,n+1]$, we have that 
\begin{dmath*}
\sum_{q \leq N_1 < \cdots <N_m \leq n+1}{a_{(m);N_m} \cdots a_{(1);N_1}} 
=\sum_{k=0}^{m}{\left(\prod_{j=0}^{m-k-1}{a_{(m-j);n+1-j}}\right)\left(\sum_{q \leq N_1 < \cdots <N_{k} \leq n-m+k}{a_{(k);N_{k}} \cdots a_{(1);N_1}} \right)}
.\end{dmath*}
Using the notation from Eq.~\eqref{eq1}, this theorem can be written as 
$$
P_{m,q,n+1}
=\sum_{k=0}^{m}{\left(\prod_{j=0}^{m-k-1}{a_{(m-j);n+1-j}}\right) P_{k,q,n-m+k}}
.$$
\end{theorem}
\begin{proof}
$ \\ $
1. Base Case: verify true for $m=1$.
\begin{equation*}
\begin{split}
\sum_{k=0}^{1}{\left(\prod_{j=0}^{-k}{a_{(1-j);n+1-j}}\right) P_{k,q,n-1+k}}
&=\left(\prod_{j=0}^{0}{a_{(1-j);n+1-j}}\right) P_{0,q,n-1}
+\left(\prod_{j=0}^{-1}{a_{(1-j);n+1-j}}\right) P_{1,q,n}\\
&=(a_{(1);n+1})(1)+(1)\sum_{q \leq N_1 \leq n}{a_{(1);n+1}}\\
&=\sum_{q \leq N_1 \leq n+1}{a_{(1);n+1}}\\
&=P_{1,q,n+1} 
.\end{split}
\end{equation*}
2. Induction hypothesis: assume the statement is true until $m$. 
$$
P_{m,q,n+1}
=\sum_{k=0}^{m}{\left(\prod_{j=0}^{m-k-1}{a_{(m-j);n+1-j}}\right) P_{k,q,n-m+k}}
.$$
3. Induction step: we will show that this statement is true for ($m+1$).\\
We have to show the following statement to be true: 
$$
P_{m+1,q,n+1}
=\sum_{k=0}^{m+1}{\left(\prod_{j=0}^{m-k}{a_{(m+1-j);n+1-j}}\right) P_{k,q,n-m-1+k}}
.$$
$ \\ $
From Lemma \ref{Lemma 2.1}, 
$$P_{m+1,q,n+1}=P_{m+1,q,n}+a_{(m+1);n+1} P_{m,q,n}.$$
By applying the induction hypothesis for $n$ instead of $(n+1)$, 
\begin{equation*}
\begin{split}
P_{m+1,q,n+1}
&=P_{m+1,q,n}+a_{(m+1);n+1}\sum_{k=0}^{m}{\left(\prod_{j=0}^{m-k-1}{a_{(m-j);n-j}}\right) P_{k,q,n-1-m+k}}\\
&=P_{m+1,q,n}+a_{(m+1);n+1}\sum_{k=0}^{m}{\left(\prod_{j=1}^{m-k}{a_{(m+1-j);n+1-j}}\right) P_{k,q,n-1-m+k}}\\
&=P_{m+1,q,n}+\sum_{k=0}^{m}{\left(\prod_{j=0}^{m-k}{a_{(m+1-j);n+1-j}}\right) P_{k,q,n-m-1+k}} 
.\end{split}
\end{equation*}
Noticing that 
$$
\sum_{k=m+1}^{m+1}{\left(\prod_{j=0}^{m-k}{a_{(m+1-j);n+1-j}}\right) P_{k,q,n-m-1+k}}
=P_{m+1,q,n}
.$$
Hence, by substituting back, we get the desired relation. \\
The case for ($m+1$) is proven. Hence, the theorem is proven by induction. 
\end{proof}
\begin{corollary} \label{Corollary 2.1}
If all sequences are the same, Theorem \ref{Theorem 2.1} becomes 
\begin{dmath*}
\sum_{q \leq N_1 < \cdots <N_m \leq n+1}{a_{N_m} \cdots a_{N_1}} 
=\sum_{k=0}^{m}{\left(\prod_{j=0}^{m-k-1}{a_{n+1-j}}\right)\left(\sum_{q \leq N_1 < \cdots <N_{k} \leq n-m+k}{a_{N_{k}} \cdots a_{N_1}} \right)}
.\end{dmath*}
Using the notation from Eq.~\eqref{eq2}, this theorem can be written as 
$$
\hat{P}_{m,q,n+1}
=\sum_{k=0}^{m}{\left(\prod_{j=0}^{m-k-1}{a_{n+1-j}}\right) \hat{P}_{k,q,n-m+k}}
.$$
\end{corollary}
\begin{example}
{\em Consider that $m=2$, we have the two following cases: }
\begin{itemize}
\item {\em If all sequences are distinct, } 
$$
\sum_{q \leq N_1 < N_2 \leq n+1}{b_{N_2}a_{N_1}}
-\sum_{q \leq N_1 < N_2 \leq n}{b_{N_2}a_{N_1}}
=(b_{n+1})\sum_{q \leq N_1 \leq n-1}{a_{N_1}}+(b_{n+1})(a_{n})
.$$
\item {\em If all sequences are the same, } 
$$
\sum_{q \leq N_1 < N_2 \leq n+1}{a_{N_2}a_{N_1}}
-\sum_{q \leq N_1 < N_2 \leq n}{a_{N_2}a_{N_1}}
=(a_{n+1})\sum_{q \leq N_1 \leq n-1}{a_{N_1}}+(a_{n+1})(a_{n})
.$$
\end{itemize}
\end{example}
\subsection{Simple recurrent expression}
Theorem \ref{Theorem 2.1} can be rewritten in a recursive way as illustrated by the following theorem.  
\begin{theorem} \label{Theorem 2.2}
For any $m,q,n \in \mathbb{N}$ where $n \geq q+m-1$ and for any set of sequences $a_{(1);N_1},\ldots,a_{(m);N_m}$ defined in the interval $[q,n+1]$, we have that 
\begin{dmath*}
\sum_{q \leq N_1 < \cdots <N_m \leq n+1}{a_{(m);N_m} \cdots a_{(1);N_1}} 
-\sum_{q \leq N_1 < \cdots <N_m \leq n}{a_{(m);N_m} \cdots a_{(1);N_1}} 
=a_{(m);n+1}\left\{a_{(m-1);n} \left[ \cdots a_{(2);n-m+3} \left( a_{(1);n-m+2}(1)+\sum_{q \leq N_1 \leq n-m+1}{a_{(1);N_1}} \right)+\sum_{q \leq N_1 <N_2 \leq n-m+2}{a_{(2);N_2} a_{(1);N_1}}\right] +\sum_{q \leq N_1 < \cdots < N_{m-1} \leq n-1}{a_{(m-1);N_{m-1}}\cdots a_{(1);N_1}} \right\}.
\end{dmath*}
Using the notation from Eq.~\eqref{eq1}, this theorem can be written as 
\begin{dmath*}
P_{m,q,n+1}
=a_{(m);n+1}\left \{ a_{(m-1);n}\left[ \cdots a_{(2);n-m+3} \left( a_{(1);n-m+2} \left( P_{0,q,n-m+1} \right) + P_{1,q,n-m+1} \right) + P_{2,q,n-m+2} \right] + P_{m-1,q,n-1} \right \} + P_{m,q,n}
\end{dmath*}
where $P_{0,q,n-m+1}=1$. 
\end{theorem}
\begin{proof}
$ \\ $
1. Base Case: verify true for $m=1$. \\
From Lemma \ref{Lemma 2.1} for $m=1$, 
$$
P_{1,q,n+1}=P_{1,q,n}+a_{(1);n+1}P_{0,q,n}
.$$
2. Induction Hypothesis: assume the statement is true until $m$. 
\begin{dmath*}
P_{m,q,n+1}
=a_{(m);n+1}\left \{ a_{(m-1);n}\left[ \cdots a_{(2);n-m+3} \left( a_{(1);n-m+2} \left( P_{0,q,n-m+1} \right) + P_{1,q,n-m+1} \right) + P_{2,q,n-m+2} \right] + P_{m-1,q,n-1} \right \} + P_{m,q,n}.
\end{dmath*}
3. Induction Step: we will show that this statement is true for $(m+1)$. \\
We have to show the following statement to be true: 
\begin{dmath*}
P_{m+1,q,n+1}
=a_{(m+1);n+1}\left \{ a_{(m);n}\left[ \cdots a_{(2);n-m+2} \left( a_{(1);n-m+1} \left( P_{0,q,n-m} \right) + P_{1,q,n-m} \right) + P_{2,q,n-m+1} \right] + P_{m,q,n-1} \right \} + P_{m+1,q,n}.
\end{dmath*}
$ \\ $ 
From Lemma \ref{Lemma 2.1}, 
$$P_{m+1,q,n+1}=P_{m+1,q,n}+a_{(m+1);n+1} P_{m,q,n}.$$ 
By applying the induction hypothesis with $n$ instead of $n+1$, we get the desired theorem, 
\begin{dmath*}
P_{m+1,q,n+1}
=a_{(m+1);n+1}\left \{ a_{(m);n}\left[ \cdots a_{(2);n-m+2} \left( a_{(1);n-m+1} \left( P_{0,q,n-m} \right) + P_{1,q,n-m} \right) + P_{2,q,n-m+1} \right] + P_{m,q,n-1} \right \} + P_{m+1,q,n}.
\end{dmath*}
The case for ($m+1$) is proven. Hence, the theorem is proven by induction. 
\end{proof}
\begin{corollary} \label{Corollary 2.2}
If all sequences are the same, Theorem \ref{Theorem 2.2} becomes 
\begin{dmath*}
\sum_{q \leq N_1 < \cdots <N_m \leq n+1}{a_{N_m} \cdots a_{N_1}} 
-\sum_{q \leq N_1 < \cdots <N_m \leq n}{a_{N_m} \cdots a_{N_1}} 
=a_{n+1}\left\{a_{n} \left[ \cdots a_{n-m+3} \left( a_{n-m+2}(1)+\sum_{q \leq N_1 \leq n-m+1}{a_{N_1}} \right)+\sum_{q \leq N_1 <N_2 \leq n-m+2}{a_{N_2} a_{N_1}}\right] +\sum_{q \leq N_1 < \cdots < N_{m-1} \leq n-1}{a_{N_{m-1}}\cdots a_{N_1}} \right\}
.\end{dmath*}
Using the notation from Eq.~\eqref{eq2}, this theorem can be written as 
\begin{dmath*}
\Scale[0.9]{
\hat{P}_{m,q,n+1}
=a_{n+1}\left \{ a_{n}\left[ \cdots a_{n-m+3} \left( a_{n-m+2} \left( \hat{P}_{0,q,n-m+1} \right) + \hat{P}_{1,q,n-m+1} \right) + \hat{P}_{2,q,n-m+2} \right] + \hat{P}_{m-1,q,n-1} \right \} + \hat{P}_{m,q,n}}
\end{dmath*}
where $\hat{P}_{0,q,n-m+1}=1$.
\end{corollary}
\begin{example}
{\em Consider that $m=2$, we have the two following cases: }
\begin{itemize}
\item {\em If all sequences are distinct,  }
$$
\sum_{q \leq N_1 < N_2 \leq n+1}{b_{N_2}a_{N_1}}
-\sum_{q \leq N_1 < N_2 \leq n}{b_{N_2}a_{N_1}}
=(b_{n+1}) \left\{ \sum_{q \leq N_1 \leq n-1}{a_{N_1}}+a_{n}(1) \right\}
.$$
\item {\em If all sequences are the same,  }
$$
\sum_{q \leq N_1 < N_2 \leq n+1}{a_{N_2}a_{N_1}}
-\sum_{q \leq N_1 < N_2 \leq n}{a_{N_2}a_{N_1}}
=(a_{n+1}) \left\{ \sum_{q \leq N_1 \leq n-1}{a_{N_1}}+a_{n}(1) \right\}
.$$
\end{itemize}
\end{example}
\subsection{General expression}
In order to represent the variation of a multiple sum of order $m$ ($P_{m,q,n+1}-P_{m,q,n}$) only in terms of multiple sums of order going from $p$ to $(m-1)$, a more general form of Theorem \ref{Theorem 2.1} can be developed.  
\begin{theorem} \label{Theorem 2.3}
For any $m,q,n \in \mathbb{N}$ where $n \geq q+m-1$, for any $p \in [0,m]$, and for any set of sequences $a_{(1);N_1},\ldots,a_{(m);N_m}$ defined in the interval $[q,n+1]$, we have that
\begin{dmath*}
\sum_{q \leq N_1 < \cdots <N_m \leq n+1}{a_{(m);N_m} \cdots a_{(1);N_1}} 
=\sum_{k=p+1}^{m}{\left(\prod_{j=0}^{m-k-1}{a_{(m-j);n+1-j}}\right)\left(\sum_{q \leq N_1 < \cdots <N_{k} \leq n-m+k}{a_{(k);N_{k}} \cdots a_{(1);N_1}} \right)}
+\left(\prod_{j=0}^{m-p-1}{a_{(m-j);n+1-j}}\right)\left(\sum_{q \leq N_1 < \cdots <N_p \leq n-m+p+1}{a_{(p);N_p} \cdots a_{(1);N_1}} \right)
.\end{dmath*}
Using the notation from Eq.~\eqref{eq1}, this theorem can be written as 
$$
P_{m,q,n+1}
=\sum_{k=p+1}^{m}{\left(\prod_{j=0}^{m-k-1}{a_{(m-j);n+1-j}}\right)P_{k,q,n-m+k}}
+\left(\prod_{j=0}^{m-p-1}{a_{(m-j);n+1-j}}\right)P_{p,q,n-m+p+1}
.$$
\end{theorem}
\begin{proof}
By applying Theorem \ref{Theorem 2.1}, 
\begin{dmath*}
P_{m,q,n+1}
=\sum_{k=0}^{m}{\left(\prod_{j=0}^{m-k-1}{a_{(m-j);n+1-j}}\right) P_{k,q,n-m+k}}
=\sum_{k=p+1}^{m}{\left(\prod_{j=0}^{m-k-1}{a_{(m-j);n+1-j}}\right) P_{k,q,n-m+k}}
+\sum_{k=0}^{p}{\left(\prod_{j=0}^{m-k-1}{a_{(m-j);n+1-j}}\right) P_{k,q,n-m+k}}
=\sum_{k=p+1}^{m}{\left(\prod_{j=0}^{m-k-1}{a_{(m-j);n+1-j}}\right) P_{k,q,n-m+k}}
+\left(\prod_{j=0}^{m-p-1}{a_{(m-j);n+1-j}}\right)\sum_{k=0}^{p}{\left(\prod_{j=m-p}^{m-k-1}{a_{(m-j);n+1-j}}\right) P_{k,q,n-m+k}}.
\end{dmath*}
From Theorem \ref{Theorem 2.1}, with $m$ substituted by $p$ and $n$ substituted by $n-m+p$, we have  
$$
P_{p,q,n-m+p+1}
=\sum_{k=0}^{p}{\left(\prod_{j=0}^{p-k-1}{a_{(p-j);n-m+p+1-j}}\right) P_{k,q,n-m+k}}
=\sum_{k=0}^{p}{\left(\prod_{j=m-p}^{m-k-1}{a_{(m-j);n+1-j}}\right) P_{k,q,n-m+k}}
.$$
Hence, by substituting, we get 
$$
P_{m,q,n+1}
=\sum_{k=p+1}^{m}{\left(\prod_{j=0}^{m-k-1}{a_{(m-j);n+1-j}}\right)P_{k,q,n-m+k}}
+\left(\prod_{j=0}^{m-p-1}{a_{(m-j);n+1-j}}\right)P_{p,q,n-m+p+1}
.$$
\end{proof}
\begin{corollary} \label{Corollary 2.3}
If all sequences are the same, Theorem \ref{Theorem 2.3} simplifies to the following,  
\begin{dmath*}
\sum_{q \leq N_1 < \cdots <N_m \leq n+1}{a_{N_m} \cdots a_{N_1}} 
=\sum_{k=p+1}^{m}{\left(\prod_{j=0}^{m-k-1}{a_{n+1-j}}\right)\left(\sum_{q \leq N_1 < \cdots <N_{k} \leq n-m+k}{a_{N_{k}} \cdots a_{N_1}} \right)}
+\left(\prod_{j=0}^{m-p-1}{a_{n+1-j}}\right)\left(\sum_{q \leq N_1 < \cdots <N_p \leq n-m+p+1}{a_{N_p} \cdots a_{N_1}} \right). 
\end{dmath*}
Using the notation from Eq.~\eqref{eq2}, this theorem can be written as 
$$
\hat{P}_{m,q,n+1}
=\sum_{k=p+1}^{m}{\left(\prod_{j=0}^{m-k-1}{a_{n+1-j}}\right)\hat{P}_{k,q,n-m+k}}
+\left(\prod_{j=0}^{m-p-1}{a_{n+1-j}}\right)\hat{P}_{p,q,n-m+p+1}
.$$
\end{corollary}
\begin{example}
{\em For $p=2$ and if the sequences are the same: }
\begin{dmath*}
\sum_{q \leq N_1< \cdots <N_m \leq n+1}{a_{N_m}\cdots a_{N_1}}
=\sum_{k=3}^{m}{\left(\prod_{j=0}^{m-k-1}{a_{n+1-j}}\right)\left(\sum_{q \leq N_1 < \cdots <N_{k} \leq n-m+k}{a_{N_{k}} \cdots a_{N_1}} \right)}
+\left(\prod_{j=0}^{m-3}{a_{n+1-j}}\right)\left(\sum_{q \leq N_1 <N_2 \leq n-m+3}{a_{N_2}a_{N_1}} \right)
.\end{dmath*}
\end{example}
\begin{example}
{\em For $p=m-2$ and if the sequences are the same: }
\begin{dmath*}
\sum_{q \leq N_1< \cdots <N_m \leq n+1}{a_{N_m}\cdots a_{N_1}}
-\sum_{q \leq N_1< \cdots <N_m \leq n}{a_{N_m}\cdots a_{N_1}}
=\left(a_{n+1}\right)\left(\sum_{q \leq N_1< \cdots <N_{m-1} \leq n-1}{a_{N_{m-1}}\cdots a_{N_1}}\right)
+\left(a_{n+1}a_{n}\right)\left(\sum_{q \leq N_1< \cdots <N_{m-2} \leq n-1}{a_{N_{m-2}}\cdots a_{N_1}}\right)
.\end{dmath*}
\end{example}
\subsection{General recurrent expression}
The general expression of the variation formula (illustrated by Theorem \ref{Theorem 2.3}) can be expressed in a recursive way as illustrated by the following theorem.  
\begin{theorem} \label{Theorem 2.4}
For any $m,q,n \in \mathbb{N}$ where $n \geq q+m-1$, for any $p \in [0,m]$, and for any set of sequences $a_{(1);N_1},\ldots,a_{(m);N_m}$ defined in the interval $[q,n+1]$, we have that 
\begin{dmath*}
P_{m,q,n+1}=a_{(m);n+1}\left\{ a_{(m-1);n}\left[ \cdots a_{(p+2);n-m+p+3} \left( a_{(p+1);n-m+p+2} \left( P_{p,q,n-m+p+1} \right) + P_{p+1,q,n-m+p+1} \right) + P_{p+2,q,n-m+p+2} \right] + P_{m-1,q,n-1} \right\} + P_{m,q,n}
.\end{dmath*}
\end{theorem}
\begin{proof}
From Theorem \ref{Theorem 2.2}, with $m$ substituted by $p$ and $n$ substituted by $n-m+p$, we have 
\begin{dmath*}
P_{p,q,n-m+p+1}
=a_{(p);n-m+p+1}\left \{ a_{(p-1);n-m+p}\left[ \cdots a_{(2);n-m+3} \left( a_{(1);n-m+2} \left( P_{0,q,n-m+1} \right) + P_{1,q,n-m+1} \right) + P_{2,q,n-m+2} \right] + P_{p-1,q,n-m+p-1} \right \} + P_{p,q,n-m+p}
\end{dmath*}
where $P_{0,q,n-m+1}=1$. \\
Substituting into the expression of Theorem \ref{Theorem 2.2}, the inner part becomes $P_{p,q,n-m+p+1}$ and we get the desired formula.  
\end{proof}
\begin{corollary} \label{Corollary 2.4}
If all sequences are the same, Theorem \ref{Theorem 2.4} simplifies to the following form, \begin{dmath*}
\hat{P}_{m,q,n+1}=a_{n+1}\left\{ a_{n}\left[ \cdots a_{n-m+p+3} \left( a_{n-m+p+2} \left( \hat{P}_{p,q,n-m+p+1} \right) + \hat{P}_{p+1,q,n-m+p+1} \right) + \hat{P}_{p+2,q,n-m+p+2} \right] + \hat{P}_{m-1,q,n-1} \right\} + \hat{P}_{m,q,n}.
\end{dmath*}
\end{corollary}
\begin{example}
{\em For $p=m-2$ and if the sequences are the same: }
\begin{dmath*}
\sum_{q \leq N_1< \cdots <N_m \leq n+1}{a_{N_m}\cdots a_{N_1}}
-\sum_{q \leq N_1< \cdots <N_m \leq n}{a_{N_m}\cdots a_{N_1}}
=a_{n+1}\left\{\sum_{q \leq N_1< \cdots <N_{m-1} \leq n-1}{a_{N_{m-1}}\cdots a_{N_1}}+a_{n} \left[ \sum_{q \leq N_1< \cdots <N_{m-2} \leq n-1}{a_{N_{m-2}}\cdots a_{N_1}}\right] \right\}
.\end{dmath*}
\end{example}
\section{Reduction Formulas} \label{Reduction Formulas}
The objective of this section is to introduce formulas which can be used to reduce multiple sums from their original form containing multiple summations $\left(\sum_{q \leq N_1 < \cdots < N_m \leq n}{a_{N_m} \cdots a_{N_1}}\right)$ to a form containing only simple sums $\left(\left(\sum_{N=q}^{n}{(a_N)^i}\right)^y\right)$. This will involve the use of partitions of an integer. 
\subsection{A brief introduction to partitions}
In this paper, partitions are involved in the reduction formula for a multiple sum. For this reason, a brief introduction to partitions will be given in this section. 
\begin{definition}
A partition of a non-negative integer m is a set of positive integers whose sum equals $m$.
We can represent a partition of $m$ as a vector $(y_{k,1},\ldots,y_{k,m})$ that verifies
\begin{align} \label{partitionDef}
\begin{pmatrix}
y_{k,1} \\
\vdots \\
y_{k,m} \\
\end{pmatrix}
\cdot
\begin{pmatrix}
1 \\
\vdots \\
m \\
\end{pmatrix}
=y_{k,1}+2y_{k,2}+ \cdots + my_{k,m}=m
.\end{align}
\end{definition}
The set $\{(y_{1,1},\ldots,y_{1,m}),(y_{2,1},\ldots,y_{2,m}), \cdots \}$ containing all of the vectors $(y_{k,1},\ldots,y_{k,m})$ that verify Eq.~\eqref{partitionDef} represents the set of all partitions of $m$. The size of this set represents the number of partitions of $m$ and is given by the partition function $p(m)$. The value of $p(m)$ can be calculated from the generating function developed, in the eighteen century, by Euler \cite{PartitionEuler}, 
\begin{equation}
\sum_{m=0}^{\infty}{p(m)x^m}=\prod_{j=1}^{\infty}{\frac{1}{1-x^j}}.
\end{equation}
Other methods for calculating $p(m)$ include the recurrent definition provided by Euler, the asymptotic expression introduced by Hardy and Ramanujan in 1918 \cite{Hardy}, and Rademacher's formula for $p(m)$ \cite{Rademacher} \cite{rademacher1943expansion}.
However, the most celebrated way of obtaining $p(m)$ is by applying the formula developed by Ono and Bruinier which expresses $p(m)$ as a finite sum \cite{Ono}.
 \\
Additionnaly, a partition can be represented using a variety of ways although the most common representations are Ferrers diagrams and Young diagrams. Similarly, there exists some variants of Ferrers diagrams that are used (see \cite{propp1989some}). 
\begin{remark}
{\em A more in-depth explanation of partitions can be found in \cite{andrews1998theory}. }
\end{remark}
\subsection{Reduction Theorem and Partition Identities}
In order to prove the main theorem of this section (Theorem \ref{Theorem 3.1}, which we will call the reduction theorem), we have to prove a set of lemmas. There are 2 sets of lemmas needed: The first set of lemmas is also needed to prove the reduction theorem for Recurrent sums and, hence, was already proven in \cite{RecurrentSums}. The second set of lemmas is specific to the type of sums studied in this paper and will be proven in this section. Additionally, note that even though the major reason for proving these lemmas is to prove the reduction theorem, however, these lemmas are important on their own as they provide relations governing partitions. \\

Before we can proceed to prove the needed lemmas, we need to define the following notation: Let $[x^r]\left(P(x)\right)$ represent the coefficient of $x^r$ in $P(x)$. Let $x^{\overline{m}}=x(x+1)\cdots (x+m-1)$ represent the rising factorial. Let $(x)_m=x(x-1)\cdots (x-m+1)$ represent the falling factorial. \\
The original definition of Stirling numbers of the first kind $S(m,r)$ was as the coefficients in the expansion of $(x)_m$:
\begin{equation}
(x)_m
=\sum_{k=0}^{m}{S(m,k)x^k} \,\,\,\,\,\,\,\, or \,\,\,\,\,\,\,\,   S(m,r)=[x^r](x)_m
.\end{equation}
In a similar way, the unsigned Stirling numbers of the first kind, denoted $|S(m,r)|$ or ${m \brack r}$, can be expressed in terms of the rising factorial $x^{\overline{m}}$:
\begin{equation}
x^{\overline{m}}
=\sum_{k=0}^{m}{{m \brack k}x^k} \,\,\,\,\,\,\,\, or \,\,\,\,\,\,\,\,   {m \brack r}=[x^r]\left(x^{\overline{m}}\right)
.\end{equation}
\begin{remark}
{\em For $m \geq 1$, ${m \brack 0}=0.$ }
\end{remark}
\begin{remark}
{\em For $m \geq 0$, ${m \brack m}=1.$ }
\end{remark}
From this definition, the famous finite alternating sum of the unsigned Stirling numbers of the first kind can be directly deduced by substituting $x$ by $(-1)$ to get 
\begin{itemize}
\item If $m=0$ or $m=1$ 
\begin{equation}
\sum_{k=0}^{0}{(-1)^k {0 \brack k}}
={0 \brack 0}=1=0!(-1)^0.
\end{equation}
\begin{equation}
\sum_{k=0}^{1}{(-1)^k {1 \brack k}}
={1 \brack 0}-{1 \brack 1}=0-1=1!(-1)^1. 
\end{equation}
\item If $m \geq 2$ 
\begin{equation}
\sum_{k=0}^{m}{(-1)^k {m \brack k}}
=(-1)(-1+1)\cdots (-1+m-1)=0. 
\end{equation}
\end{itemize}
Hence, 
\begin{equation} \label{StirlingSum}
\sum_{k=0}^{m}{(-1)^k {m \brack k}}=
\begin{cases}
(-1)^m m! & $for $ 0 \leq m \leq 1, \\
0 & $for $ m \geq 2. \\
\end{cases}
\end{equation}
\begin{remark}
{\em More details on Stirling numbers of the first kind can be found in  \cite{loeb1992generalization}. }
\end{remark}
For simplicity, we define $\sum{f(i)}$ to represent $\sum_{i=1}^{m}{f(i)}$. In particular, $\sum{i.y_{k,i}}=\sum_{i=1}^{m}{i.y_{k,i}}$ and $\sum{y_{k,i}}=\sum_{i=1}^{m}{y_{k,i}}$. Additionally, a partition of $m$ of length $r$ is a partition of $m$ where $\sum{y_{k,i}}=r$. \\

Now that we have defined the needed notation, we can start proving the required lemmas. We begin by proving the following identity involving an alternating sum over partitions. 
\begin{lemma} \label{Lemma 3.1}
Let $m$ be a non-negative integer, 
\begin{equation*}
\sum_{\substack{k \\ \sum{i.y_{k,i}}=m}}{\prod_{i=1}^{m}{\frac{(-1)^{y_{k,i}}}{i^{y_{k,i}}(y_{k,i})!}}}
=\sum_{\substack{k \\ \sum{i.y_{k,i}}=m}}{(-1)^{\sum{{y_{k,i}}}}\prod_{i=1}^{m}{\frac{1}{i^{y_{k,i}}(y_{k,i})!}}}
=
\begin{cases}
(-1)^m & $ {\em for} $0 \leq m \leq 1, \\
0 & $ {\em for} $ m \geq 2. \\
\end{cases}
\end{equation*}
\end{lemma}
\begin{proof} 
\begin{equation*}
\begin{split}
\sum_{\substack{k \\ \sum{i.y_{k,i}}=m}}{\prod_{i=1}^{m}{\frac{(-1)^{y_{k,i}}}{i^{y_{k,i}}(y_{k,i})!}}}
&=\sum_{r=0}^{m}{\sum_{\substack{k \\ \sum{i.y_{k,i}}=m \\ \sum{y_{k,i}}=r}}{(-1)^{\sum{y_{k,i}}}\prod_{i=1}^{m}{\frac{1}{i^{y_{k,i}}(y_{k,i})!}}}} \\
&=\sum_{r=0}^{m}{(-1)^r \sum_{\substack{k \\ \sum{i.y_{k,i}}=m \\ \sum{y_{k,i}}=r}}{\prod_{i=1}^{m}{\frac{1}{i^{y_{k,i}}(y_{k,i})!}}}}
.\end{split}
\end{equation*}
From \cite{RecurrentSums}, we have 
$$
\sum_{\substack{k \\ \sum{i.y_{k,i}}=m \\ \sum{y_{k,i}}=r}}{\prod_{i=1}^{m}{\frac{1}{i^{y_{k,i}}(y_{k,i})!}}}
=\frac{1}{m!}{m \brack r}
.$$
Hence, 
$$
\sum_{\substack{k \\ \sum{i.y_{k,i}}=m}}{\prod_{i=1}^{m}{\frac{(-1)^{y_{k,i}}}{i^{y_{k,i}}(y_{k,i})!}}}
=\sum_{r=0}^{m}{(-1)^r \frac{1}{m!}{m \brack r}}
=\frac{1}{m!}\sum_{r=0}^{m}{(-1)^r {m \brack r}}
.$$
Using Eq.~\eqref{StirlingSum}, we obtain the desired theorem. 
\end{proof}
A more general form of Lemma \ref{Lemma 3.1} is illustrated in the following lemma.  
\begin{lemma} \label{Lemma 3.2}
Let $(y_{k,1}, \cdots , y_{k,m})=\{(y_{1,1}, \cdots , y_{1,m}),(y_{2,1}, \cdots , y_{2,m}), \cdots \}$ be the set of all partitions of $m$. Let $(\varphi_1, \cdots , \varphi_m)$ be a partition of $r \leq m$.
\begin{dmath*}
\sum_{\substack{k \\ \sum{i.y_{k,i}}=m}}{\prod_{i=1}^{m}{\frac{(-1)^{y_{k,i}}\binom{y_{k,i}}{\varphi_i}}{i^{y_{k,i}}(y_{k,i})!}}}
=\sum_{\substack{k \\ \sum{i.y_{k,i}}=m \\ y_{k,i} \geq \varphi_i}}{\prod_{i=1}^{m}{\frac{(-1)^{y_{k,i}}\binom{y_{k,i}}{\varphi_i}}{i^{y_{k,i}}(y_{k,i})!}}}
=
\begin{cases}
(-1)^{m-r} \prod_{i=1}^{m}{\frac{(-1)^{\varphi_i}}{i^{\varphi_{i}} (\varphi_{i})!}} & ${\em for} $ 0 \leq m-r \leq 1, \\
0 & ${\em for} $ m-r \geq 2. \\
\end{cases}
\end{dmath*}
\begin{remark}
{\em Knowing that the largest element of a partition of $r$ is $r$, we can rewrite it as follows }
\begin{dmath*}
\sum_{\substack{k \\ \sum{i.y_{k,i}}=m}}{\prod_{i=1}^{m}{\frac{(-1)^{y_{k,i}}\binom{y_{k,i}}{\varphi_i}}{i^{y_{k,i}}(y_{k,i})!}}}
=\sum_{\substack{k \\ \sum{i.y_{k,i}}=m \\ y_{k,i} \geq \varphi_i}}{\prod_{i=1}^{m}{\frac{(-1)^{y_{k,i}}\binom{y_{k,i}}{\varphi_i}}{i^{y_{k,i}}(y_{k,i})!}}}
=
\begin{cases}
(-1)^{m-r} \prod_{i=1}^{r}{\frac{(-1)^{\varphi_i}}{i^{\varphi_{i}} (\varphi_{i})!}} & ${\em for} $ 0 \leq m-r \leq 1, \\
0 & ${\em for} $ m-r \geq 2. \\
\end{cases}
\end{dmath*}
\end{remark}
\end{lemma}
\begin{proof}
Knowing that $\binom{n}{k}$ is zero if $n<k$, then $\binom{y_{k,i}}{\varphi_i}=0$ if $ \exists i \in \mathbb{N}, y_{k,i}<\varphi_i$. Hence, 
\begin{dmath*}
\sum_{\substack{k \\ \sum{i.y_{k,i}}=m}}{\prod_{i=1}^{m}{\frac{(-1)^{y_{k,i}}\binom{y_{k,i}}{\varphi_i}}{i^{y_{k,i}}(y_{k,i})!}}}
=\sum_{\substack{k \\ \sum{i.y_{k,i}}=m \\ \exists i, y_{k,i}<\varphi_i}}{\prod_{i=1}^{m}{\frac{(-1)^{y_{k,i}}\binom{y_{k,i}}{\varphi_i}}{i^{y_{k,i}}(y_{k,i})!}}}
+\sum_{\substack{k \\ \sum{i.y_{k,i}}=m \\ y_{k,i} \geq \varphi_i}}{\prod_{i=1}^{m}{\frac{(-1)^{y_{k,i}}\binom{y_{k,i}}{\varphi_i}}{i^{y_{k,i}}(y_{k,i})!}}}
=\sum_{\substack{k \\ \sum{i.y_{k,i}}=m \\ y_{k,i} \geq \varphi_i}}{\prod_{i=1}^{m}{\frac{(-1)^{y_{k,i}}\binom{y_{k,i}}{\varphi_i}}{i^{y_{k,i}}(y_{k,i})!}}}
.\end{dmath*}
The first part of the proof is complete. 
\begin{equation*}
\begin{split}
\sum_{\substack{k \\ \sum{i.y_{k,i}}=m}}{\prod_{i=1}^{m}{\frac{(-1)^{y_{k,i}}\binom{y_{k,i}}{\varphi_i}}{i^{y_{k,i}}(y_{k,i})!}}}
&=\sum_{\substack{k \\ \sum{i.y_{k,i}}=m}}{\prod_{i=1}^{m}{\frac{(-1)^{y_{k,i}}}{i^{y_{k,i}}(y_{k,i})!}.\frac{y_{k,i}!}{\varphi_i! (y_{k,i}-\varphi_i)!}}}\\
&=\sum_{\substack{k \\ \sum{i.y_{k,i}}=m}}{\prod_{i=1}^{m}{\frac{(-1)^{y_{k,i}}}{i^{y_{k,i}}}.\frac{1}{\varphi_i! (y_{k,i}-\varphi_i)!}}} \\
&=\sum_{\substack{k \\ \sum{i.y_{k,i}}=m}}{\prod_{i=1}^{m}{\frac{(-1)^{\varphi_{i}}}{i^{\varphi_i} \varphi_i!}.\frac{(-1)^{y_{k,i}-\varphi_{i}}}{i^{y_{k,i}-\varphi_i} (y_{k,i}-\varphi_i)!}}} \\
&=\sum_{\substack{k \\ \sum{i.y_{k,i}}=m}}{\prod_{i=1}^{m}{\frac{(-1)^{\varphi_{i}}}{i^{\varphi_i} \varphi_i!}}\prod_{i=1}^{m}{\frac{(-1)^{y_{k,i}-\varphi_{i}}}{i^{y_{k,i}-\varphi_i} (y_{k,i}-\varphi_i)!}}} 
.\end{split}
\end{equation*}
As $\varphi_1, \ldots , \varphi_m$ are all constants then $\prod_{i=1}^{m}{\frac{(-1)^{\varphi_{i}}}{i^{\varphi_i} \varphi_i!}}$ is constant. This factor is constant and is common to all terms of the sum, therefore, we can factor it and take it outside the sum. 
$$
\sum_{\substack{k \\ \sum{i.y_{k,i}}=m}}{\prod_{i=1}^{m}{\frac{(-1)^{y_{k,i}}\binom{y_{k,i}}{\varphi_i}}{i^{y_{k,i}}(y_{k,i})!}}}
=\left(\prod_{i=1}^{m}{\frac{(-1)^{\varphi_{i}}}{i^{\varphi_i} \varphi_i!}}\right) \sum_{\substack{k \\ \sum{i.y_{k,i}}=m}}{\prod_{i=1}^{m}{\frac{(-1)^{y_{k,i}-\varphi_{i}}}{i^{y_{k,i}-\varphi_i} (y_{k,i}-\varphi_i)!}}}
.$$ 
Having that $(\varphi_1, \cdots , \varphi_m)$ is a partition of $r \leq m$, hence, $\sum{i.\varphi_i}=r \leq m$. Thus, the condition $\sum{i.y_{k,i}}=m$ can be replaced by $\sum{i.(y_{k,i}-\varphi_i)}=\sum{i.y_{k,i}}-\sum{i.\varphi_i}=m-r(\geq 0)$. Hence, 
$$
\sum_{\substack{k \\ \sum{i.y_{k,i}}=m}}{\prod_{i=1}^{m}{\frac{(-1)^{y_{k,i}}\binom{y_{k,i}}{\varphi_i}}{i^{y_{k,i}}(y_{k,i})!}}}
=\left(\prod_{i=1}^{m}{\frac{(-1)^{\varphi_{i}}}{i^{\varphi_i} \varphi_i!}}\right) \sum_{\substack{k \\ \sum{i.(y_{k,i}-\varphi_i)}=m-r}}{\prod_{i=1}^{m}{\frac{(-1)^{y_{k,i}-\varphi_i}}{i^{y_{k,i}-\varphi_i} (y_{k,i}-\varphi_i)!}}}
.$$ 
Let $Y_{k,i}=y_{k,i}-\varphi_i$, 
$$
\sum_{\substack{k \\ \sum{i.y_{k,i}}=m}}{\prod_{i=1}^{m}{\frac{(-1)^{y_{k,i}}\binom{y_{k,i}}{\varphi_i}}{i^{y_{k,i}}(y_{k,i})!}}}
=\left(\prod_{i=1}^{m}{\frac{(-1)^{\varphi_i}}{i^{\varphi_i} \varphi_i!}}\right) \sum_{\substack{k \\ \sum{i.Y_{k,i}}=m-r}}{\prod_{i=1}^{m}{\frac{(-1)^{Y_{k,i}}}{i^{Y_{k,i}} Y_{k,i}!}}}
.$$ 
Knowing that the largest element of a partition of $(m-r)$ is $(m-r)$, hence,   
$$
\sum_{\substack{k \\ \sum{i.y_{k,i}}=m}}{\prod_{i=1}^{m}{\frac{(-1)^{y_{k,i}}\binom{y_{k,i}}{\varphi_i}}{i^{y_{k,i}}(y_{k,i})!}}}
=\left(\prod_{i=1}^{m}{\frac{(-1)^{\varphi_i}}{i^{\varphi_i} \varphi_i!}}\right) \sum_{\substack{k \\ \sum{i.Y_{k,i}}=m-r}}{\prod_{i=1}^{m-r}{\frac{(-1)^{Y_{k,i}}}{i^{Y_{k,i}} Y_{k,i}!}}}
.$$ 
Applying Lemma \ref{Lemma 3.1}, with $y_{k,i}$ substituted by $Y_{k,i}$ and $m$ substituted by $m-r$, we get 
\begin{dmath*}
\sum_{\substack{k \\ \sum{i.y_{k,i}}=m}}{\prod_{i=1}^{m}{\frac{(-1)^{y_{k,i}}\binom{y_{k,i}}{\varphi_i}}{i^{y_{k,i}}(y_{k,i})!}}}
=
\begin{cases}
(-1)^{m-r} \prod_{i=1}^{m}{\frac{(-1)^{\varphi_i}}{i^{\varphi_{i}} (\varphi_{i})!}} & $for $ 0 \leq m-r \leq 1, \\
0 & $for $ m-r \geq 2. \\
\end{cases}
\end{dmath*}
The proof is complete. 
\end{proof}
\begin{remark}
{\em If $r>m$, then $\sum{i.Y_{k,i}}=m-r<0$ which makes Lemma \ref{Lemma 3.1} invalid. Hence, this lemma is invalid for $r>m$. }
\end{remark}
Now that all the required lemmas have been proven, we show the following theorem which allows the representation of a multiple sum in terms of simple sums.  
\begin{theorem}[Reduction Theorem] \label{Theorem 3.1}
Let $m$ be a non-negative integer, $p(m)$ be the number of partition of $m$, $k$ be the index of the $k$-th partition of $m$ $(1 \leq k \leq p(m))$, $i$ be an integer between $1$ and $m$, and $y_{k,i}$ be the multiplicity of $i$ in the $k$-th partition of $m$. The reduction theorem for multiple sums is stated as follow:
$$\sum_{q \leq N_1 < \cdots <N_m \leq n}{a_{N_m}\cdots a_{N_1}}
=(-1)^m\sum_{\substack{k \\ \sum{i.y_{k,i}}=m}}{\prod_{i=1}^{m}{\frac{(-1)^{y_{k,i}}}{(y_{k,i})!} \left( \frac{1}{i}\sum_{N=q}^{n}{(a_N)^i }\right)^{y_{k,i}}}}.$$
\end{theorem}
\begin{remark}
{\em The theorem can also be written as }
$$\sum_{q \leq N_1 < \cdots <N_m \leq n}{a_{N_m}\cdots a_{N_1}}
=\sum_{\substack{k \\ \sum{i.y_{k,i}}=m}}{(-1)^{m-\sum{y_{k,i}}}\prod_{i=1}^{m}{\frac{1}{(y_{k,i})!} \left( \frac{1}{i}\sum_{N=q}^{n}{(a_N)^i }\right)^{y_{k,i}}}}.$$
\end{remark}
\begin{proof}
1. Base Case: verify true for $n=q$, $\forall m \in \mathbb{N}$.
\begin{equation*}
\begin{split}
(-1)^m\sum_{\sum{i.y_{k,i}}=m}{\prod_{i=1}^{m}{\frac{(-1)^{y_{k,i}}}{(y_{k,i})!} \left( \frac{1}{i}\sum_{N=q}^{q}{(a_N)^i }\right)^{y_{k,i}}}}
&=(-1)^m\sum_{\sum{i.y_{k,i}}=m}{\prod_{i=1}^{m}{\frac{(-1)^{y_{k,i}}}{(y_{k,i})!i^{y_{k,i}}} \left(a_q\right)^{i.y_{k,i}}}} \\
&=(-1)^m\sum_{\sum{i.y_{k,i}}=m}{\left(a_q\right)^{\sum{i.y_{k,i}}}\prod_{i=1}^{m}{\frac{(-1)^{y_{k,i}}}{(y_{k,i})!i^{y_{k,i}}}}} \\
&=\left(-a_q\right)^{m}\sum_{\sum{i.y_{k,i}}=m}{\prod_{i=1}^{m}{\frac{(-1)^{y_{k,i}}}{(y_{k,i})!i^{y_{k,i}}} }} 
.\end{split}
\end{equation*}
By applying Lemma \ref{Lemma 3.1}, we get 
$$
(-1)^m\sum_{\sum{i.y_{k,i}}=m}{\prod_{i=1}^{m}{\frac{(-1)^{y_{k,i}}}{(y_{k,i})!} \left( \frac{1}{i}\sum_{N=q}^{q}{(a_N)^i }\right)^{y_{k,i}}}}
=
\begin{cases}
(-a_q)^0(-1)^0=1 & $for $ m=0, \\
(-a_q)^1(-1)^1=a_q & $for $ m=1, \\
(-a_q)^m(0)=0 & $for $ m \geq 2.
\end{cases}
$$
Likewise, 
$$
\sum_{q \leq N_1 < \cdots <N_m \leq q}{a_{N_m}\cdots a_{N_1}}
=
\begin{cases}
1 & $for $ m=0, \\
\sum_{q \leq N_1 \leq q}{a_{N_1}}=a_q & $for $ m=1, \\
\sum_{q \leq N_1 < \cdots <N_m \leq q}{a_{N_m}\cdots a_{N_1}}=0 & $for $ m \geq 2.
\end{cases}
$$
2. Induction hypothesis: assume the statement is true until $n$, $\forall m \in \mathbb{N}$.
$$\sum_{q \leq N_1 < \cdots <N_m \leq n}{a_{N_m}\cdots a_{N_1}}
=(-1)^m\sum_{\sum{i.y_{k,i}}=m}{\prod_{i=1}^{m}{\frac{(-1)^{y_{k,i}}}{(y_{k,i})!} \left( \frac{1}{i}\sum_{N=q}^{n}{(a_N)^i }\right)^{y_{k,i}}}}
.$$
3. Induction step: we will show that this statement is true for $(n+1)$, $\forall m \in \mathbb{N}$. \\
We have to show the following statement to be true:  
$$
\sum_{q \leq N_1 < \cdots <N_m \leq n+1}{a_{N_m}\cdots a_{N_1}}
=(-1)^m\sum_{\sum{i.y_{k,i}}=m}{\prod_{i=1}^{m}{\frac{(-1)^{y_{k,i}}}{(y_{k,i})!} \left( \frac{1}{i}\sum_{N=q}^{n+1}{(a_N)^i }\right)^{y_{k,i}}}}
.$$
$ \\ $
\begin{dmath*}
(-1)^m\sum_{\sum{i.y_{k,i}}=m}{\prod_{i=1}^{m}{\frac{(-1)^{y_{k,i}}}{(y_{k,i})!} \left( \frac{1}{i}\sum_{N=q}^{n+1}{(a_N)^i }\right)^{y_{k,i}}}}
=(-1)^m\sum_{\sum{i.y_{k,i}}=m}{\prod_{i=1}^{m}{\frac{(-1)^{y_{k,i}}}{(y_{k,i})!i^{y_{k,i}}} \left(\sum_{N=q}^{n}{(a_N)^i}+(a_{n+1})^i\right)^{y_{k,i}}}}
.\end{dmath*}
The binomial theorem states that 
$$
(u+v)^n=\sum_{\varphi=0}^{n}{\binom{n}{\varphi}u^{n-\varphi}v^{\varphi}}
.$$
Hence, 
\begin{equation*}
\begin{split}
\left(\sum_{N=q}^{n}{(a_N)^i}+(a_{n+1})^i\right)^{y_{k,i}}
&=\sum_{\varphi_{}=0}^{y_{k,i}}{\binom{y_{k,i}}{\varphi_{}}{\left(\sum_{N=q}^{n}{(a_N)^i}\right)}^{\varphi_{}}{\left((a_{n+1})^i\right)}^{y_{k,i}-\varphi_{}}}
.\end{split}
\end{equation*}
Thus, 
\begin{dmath*}
(-1)^m\sum_{\sum{i.y_{k,i}}=m}{\prod_{i=1}^{m}{\frac{(-1)^{y_{k,i}}}{(y_{k,i})!} \left( \frac{1}{i}\sum_{N=q}^{n+1}{(a_N)^i }\right)^{y_{k,i}}}}
=(-1)^m\sum_{\sum{i.y_{k,i}}=m}{\prod_{i=1}^{m}{\frac{(-1)^{y_{k,i}}}{(y_{k,i})!i^{y_{k,i}}}\sum_{\varphi_{}=0}^{y_{k,i}}{\binom{y_{k,i}}{\varphi_{}}{\left(\sum_{N=q}^{n}{(a_N)^i}\right)}^{\varphi_{}}{\left((a_{n+1})^i\right)}^{y_{k,i}-\varphi_{}}}}}
=(-1)^m\sum_{\sum{i.y_{k,i}}=m}{\prod_{i=1}^{m}{\sum_{\varphi_{}=0}^{y_{k,i}}{\frac{(-1)^{y_{k,i}}}{(y_{k,i})!i^{y_{k,i}}}\binom{y_{k,i}}{\varphi_{}}{\left(\sum_{N=q}^{n}{(a_N)^i}\right)}^{\varphi_{}}{\left(a_{n+1}\right)}^{i.y_{k,i}-i.\varphi_{}}}}}
.\end{dmath*}
Let $A_{\varphi_{},i,k}={\frac{(-1)^{y_{k,i}}}{(y_{k,i})!i^{y_{k,i}}}\binom{y_{k,i}}{\varphi_{}}{\left(\sum_{N=q}^{n}{(a_N)^i}\right)}^{\varphi_{}}{\left(a_{n+1}\right)}^{i.y_{k,i}-i.\varphi_{}}}$.
By expanding then regrouping, it can be seen that 
$$
\prod_{i=1}^{m}{\sum_{\varphi_{}=0}^{y_{k,i}}{A_{\varphi_{},i,k}}}
=\sum_{\varphi_{m}=0}^{y_{k,m}} {\cdots \sum_{\varphi_{1}=0}^{y_{k,1}}{\prod_{i=1}^{m}{A_{\varphi_{i},i,k}}}}
=\sum_{0 \leq \varphi_{i} \leq y_{k,i}}{\prod_{i=1}^{m}{A_{\varphi_{i},i,k}}}
.$$
This is because, for any given $k$, by expanding the product of sums (the left hand side term), we will get a sum of products of the form $A_{\varphi_1,1}A_{\varphi_2,2} \cdots A_{\varphi_m,m}$ ($\prod_{i=1}^{m}{A_{\varphi_i,i}}$) for all combinations of $\varphi_1,\varphi_2, \ldots ,\varphi_m$ such that $0 \leq \varphi_1 \leq y_{k,1}, \cdots ,0 \leq \varphi_m \leq y_{k,m}$, which is equivalent to the right hand side term. 
This then can be written more compactly by expressing the repeated sum over the $\varphi_i$'s with one sum that combines all the conditions. The set of conditions $0\leq \varphi_1 \leq y_{k,1}, \ldots ,0\leq \varphi_m \leq y_{k,m}$ can be expressed as the condition $0 \leq \varphi_i \leq y_{k,i}$ for $i \in [1,m]$. \\
Hence, 
\begin{dmath*}
(-1)^m\sum_{\sum{i.y_{k,i}}=m}{\prod_{i=1}^{m}{\frac{(-1)^{y_{k,i}}}{(y_{k,i})!} \left( \frac{1}{i}\sum_{N=q}^{n+1}{(a_N)^i }\right)^{y_{k,i}}}}
=(-1)^m\sum_{\sum{i.y_{k,i}}=m}{\sum_{0 \leq \varphi_{i} \leq y_{k,i}}{\prod_{i=1}^{m}{\frac{(-1)^{y_{k,i}}}{(y_{k,i})!i^{y_{k,i}}}\binom{y_{k,i}}{\varphi_{i}}{\left(\sum_{N=q}^{n}{(a_N)^i}\right)}^{\varphi_{i}}{\left(a_{n+1}\right)}^{i.y_{k,i}-i.\varphi_{i}}}}}
.\end{dmath*}
Similarly, let $j$ represent $\sum{i.\varphi_{i}}$. Hence, we can add the trivial condition that is $j=\sum{i.\varphi_{i}}$ to the sum over $\varphi_i$. Additionally, \\
$\sum{i.\varphi_{i}}=j$ is minimal when $\varphi_1=0, \ldots , \varphi_m=0$. Hence $j_{min}=0$. \\  
$\sum{i.\varphi_{i}}=j$ is maximal when $\varphi_1=y_{k,1}, \ldots , \varphi_m=y_{k,m}$. Hence $j_{max}=\sum{i.y_{k,i}}=m$. \\ 
Therefore, we have that $0 \leq j \leq m$ or equivalently that $j$ can go from $0$ to $m$. Hence, knowing that adding a true statement to a condition does not change the condition, we can add this additional condition to get 
\begin{dmath*}
(-1)^m\sum_{\sum{i.y_{k,i}}=m}{\prod_{i=1}^{m}{\frac{(-1)^{y_{k,i}}}{(y_{k,i})!} \left( \frac{1}{i}\sum_{N=q}^{n+1}{(a_N)^i }\right)^{y_{k,i}}}}
=(-1)^m\sum_{\sum{i.y_{k,i}}=m}{\sum_{\substack{j=0 \\ \sum {i.\varphi_i}=j \\ 0 \leq \varphi_i \leq y_{k,i}}}^{m}{ \prod_{i=1}^{m}{\frac{(-1)^{y_{k,i}}}{(y_{k,i})!i^{y_{k,i}}}\binom{y_{k,i}}{\varphi_{i}}{\left(\sum_{N=q}^{n}{(a_N)^i}\right)}^{\varphi_{i}}{\left(a_{n+1}\right)}^{i.y_{k,i}-i.\varphi_{i}}}}}
.\end{dmath*}
Knowing that $\binom{y_{k,i}}{\varphi_i}=0$ if $\varphi_i>y_{k,i}$, hence, the terms produced for $\varphi_i>y_{k,i}$ would be zero. Thus, we can remove the condition $0 \leq \varphi_i \leq y_{k,i}$ because terms that do not satisfy this condition will be zeros and, therefore, would not change the value of the sum. 
\begin{dmath*}
(-1)^m\sum_{\sum{i.y_{k,i}}=m}{\prod_{i=1}^{m}{\frac{(-1)^{y_{k,i}}}{(y_{k,i})!} \left( \frac{1}{i}\sum_{N=q}^{n+1}{(a_N)^i }\right)^{y_{k,i}}}}
=(-1)^m\sum_{\sum{i.y_{k,i}}=m}{\sum_{\substack{j=0 \\ \sum {i.\varphi_i}=j}}^{m}{ \prod_{i=1}^{m}{\frac{(-1)^{y_{k,i}}}{(y_{k,i})!i^{y_{k,i}}}\binom{y_{k,i}}{\varphi_{i}}{\left(\sum_{N=q}^{n}{(a_N)^i}\right)}^{\varphi_{i}}{\left(a_{n+1}\right)}^{i.y_{k,i}-i.\varphi_{i}}}}}
.\end{dmath*}
We expand the expression then, from all values of $k$ (from every partitions $(y_{k,1}, \cdots , y_{k,m})$ of $m$), we regroup together the terms having a combination of exponents $(\varphi_1, \cdots , \varphi_m)$ that forms a partition of the same integer $j$ and we do so $\forall j \in [0,m]$. 
Hence, performing this manipulation allows us to interchange the sum over $k$ (over $\sum{i.y_{k,i}}=m$) with the sums over $j$.
Thus, the expression becomes as follows, 
\begin{dmath*}
(-1)^m\sum_{\sum{i.y_{k,i}}=m}{\prod_{i=1}^{m}{\frac{(-1)^{y_{k,i}}}{(y_{k,i})!} \left( \frac{1}{i}\sum_{N=q}^{n+1}{(a_N)^i }\right)^{y_{k,i}}}}
=(-1)^m\sum_{\substack{j=0 \\ \sum {i.\varphi_i}=j}}^{m}{\sum_{\sum{i.y_{k,i}}=m}{ \prod_{i=1}^{m}{\frac{(-1)^{y_{k,i}}}{(y_{k,i})!i^{y_{k,i}}}\binom{y_{k,i}}{\varphi_{i}}{\left(\sum_{N=q}^{n}{(a_N)^i}\right)}^{\varphi_{i}}{\left(a_{n+1}\right)}^{i.y_{k,i}-i.\varphi_{i}}}}}
=(-1)^m\sum_{\substack{j=0 \\ \sum {i.\varphi_i}=j}}^{m}{\sum_{\sum{i.y_{k,i}}=m}{{\left(a_{n+1}\right)}^{\sum{i.y_{k,i}}-\sum{i.\varphi_{i}}}\left[\prod_{i=1}^{m}{{\left(\sum_{N=q}^{n}{(a_N)^i}\right)}^{\varphi_{i}}}\right]\left[\prod_{i=1}^{m}{\frac{(-1)^{y_{k,i}}}{(y_{k,i})!i^{y_{k,i}}}\binom{y_{k,i}}{\varphi_{i}}}\right]}}
=(-1)^m\sum_{\substack{j=0 \\ \sum {i.\varphi_i}=j}}^{m}{{\left(a_{n+1}\right)}^{m-j}\left[\prod_{i=1}^{m}{{\left(\sum_{N=q}^{n}{(a_N)^i}\right)}^{\varphi_{i}}}\right]\left(\sum_{\sum{i.y_{k,i}}=m}{\prod_{i=1}^{m}{\frac{(-1)^{y_{k,i}}}{(y_{k,i})!i^{y_{k,i}}}\binom{y_{k,i}}{\varphi_{i}}}}\right)}
.\end{dmath*}
Applying Lemma \ref{Lemma 3.2}, we have 
\begin{dmath*}
\sum_{\sum{i.y_{k,i}}=m}{\prod_{i=1}^{m}{\frac{(-1)^{y_{k,i}}}{(y_{k,i})!i^{y_{k,i}}}\binom{y_{k,i}}{\varphi_{i}}}}
=
\begin{cases}
(-1)^{m-j} \prod_{i=1}^{m}{\frac{(-1)^{\varphi_i}}{i^{\varphi_{i}} (\varphi_{i})!}} & $for $ 0 \leq m-j \leq 1, \\
0 & $for $ m-j \geq 2, \\
\end{cases}
=
\begin{cases}
(-1)^{m-j} \prod_{i=1}^{m}{\frac{(-1)^{\varphi_i}}{i^{\varphi_{i}} (\varphi_{i})!}} & $for $ m-1 \leq j \leq m, \\
0 & $for $ 0 \leq j \leq m-2. \\
\end{cases}
\end{dmath*}
Hence, 
\begin{dmath*}
(-1)^m\sum_{\sum{i.y_{k,i}}=m}{\prod_{i=1}^{m}{\frac{(-1)^{y_{k,i}}}{(y_{k,i})!} \left( \frac{1}{i}\sum_{N=q}^{n+1}{(a_N)^i }\right)^{y_{k,i}}}}
=(-1)^m\sum_{\substack{j=m-1 \\ \sum {i.\varphi_i}=j}}^{m}{{\left(a_{n+1}\right)}^{m-j}\left[\prod_{i=1}^{m}{{\left(\sum_{N=q}^{n}{(a_N)^i}\right)}^{\varphi_{i}}}\right](-1)^{m-j}\left(\prod_{i=1}^{m}{\frac{(-1)^{\varphi_i}}{i^{\varphi_{i}} (\varphi_{i})!}}\right)}
=\sum_{\substack{j=m-1 \\ \sum {i.\varphi_i}=j}}^{m}{{\left(a_{n+1}\right)}^{m-j}(-1)^{j}\left(\prod_{i=1}^{m}{\frac{(-1)^{\varphi_i}}{i^{\varphi_{i}} (\varphi_{i})!}}{\left(\sum_{N=q}^{n}{(a_N)^i}\right)}^{\varphi_{i}}\right)}
.\end{dmath*} 
Knowing that for any given value of $j$ there is  multiple combinations of $\varphi_1, \ldots , \varphi_m$ that satisfy $\sum{i.\varphi_i}=j$. Hence, every value of $j$ corresponds to a sum of the sum's argument for all partitions of $j$ (for all combinations of $\varphi_1, \ldots , \varphi_m$ satisfying $\sum{i.\varphi_i}=j$). Therefore, we can split the outer sum with two conditions into two sums each with one of the conditions as follows, 
\begin{dmath*}
(-1)^m\sum_{\sum{i.y_{k,i}}=m}{\prod_{i=1}^{m}{\frac{(-1)^{y_{k,i}}}{(y_{k,i})!} \left( \frac{1}{i}\sum_{N=q}^{n+1}{(a_N)^i }\right)^{y_{k,i}}}}
=\sum_{j=m-1}^{m}{{\left(a_{n+1}\right)}^{m-j}(-1)^{j}\sum_{\sum {i.\varphi_i}=j}{\left(\prod_{i=1}^{m}{\frac{(-1)^{\varphi_i}}{i^{\varphi_{i}} (\varphi_{i})!}}{\left(\sum_{N=q}^{n}{(a_N)^i}\right)}^{\varphi_{i}}\right)}}
.\end{dmath*} 
Knowing that the largest element of a partition of $j$ is $j$, 
\begin{dmath*}
(-1)^m\sum_{\sum{i.y_{k,i}}=m}{\prod_{i=1}^{m}{\frac{(-1)^{y_{k,i}}}{(y_{k,i})!} \left( \frac{1}{i}\sum_{N=q}^{n+1}{(a_N)^i }\right)^{y_{k,i}}}}
=\sum_{j=m-1}^{m}{{\left(a_{n+1}\right)}^{m-j}(-1)^{j}\sum_{\sum {i.\varphi_i}=j}{\left(\prod_{i=1}^{j}{\frac{(-1)^{\varphi_i}}{i^{\varphi_{i}} (\varphi_{i})!}}{\left(\sum_{N=q}^{n}{(a_N)^i}\right)}^{\varphi_{i}}\right)}}
.\end{dmath*} 
By using the induction hypothesis, the expression becomes 
\begin{equation*}
\begin{split}
&(-1)^m\sum_{\sum{i.y_{k,i}}=m}{\prod_{i=1}^{m}{\frac{(-1)^{y_{k,i}}}{(y_{k,i})!} \left( \frac{1}{i}\sum_{N=q}^{n+1}{(a_N)^i }\right)^{y_{k,i}}}} \\
&\,\,=\sum_{j=m-1}^{m}{\left(a_{n+1}\right)^{m-j}\left(\sum_{q \leq N_1 < \cdots < N_j \leq n}{a_{N_j} \cdots a_{N_1}}\right)} \\
&\,\,=\left(\sum_{q \leq N_1 < \cdots < N_m \leq n}{a_{N_m} \cdots a_{N_1}}\right)
+\left(a_{n+1}\right)\left(\sum_{q \leq N_1 < \cdots < N_{m-1} \leq n}{a_{N_{m-1}} \cdots a_{N_1}}\right)
\end{split}
.\end{equation*}
Using  Lemma \ref{Lemma 2.1}, we get 
$$
(-1)^m\sum_{\sum{i.y_{k,i}}=m}{\prod_{i=1}^{m}{\frac{(-1)^{y_{k,i}}}{(y_{k,i})!} \left( \frac{1}{i}\sum_{N=q}^{n+1}{(a_N)^i }\right)^{y_{k,i}}}}
=\sum_{q \leq N_1 < \cdots <N_m \leq n+1}{a_{N_m}\cdots a_{N_1}}
.$$
The theorem is proven by induction. 
\end{proof}
\begin{corollary} \label{Corollary 3.1}
If the multiple sum starts at 1, Theorem \ref{Theorem 3.1} becomes 
$$\sum_{1 \leq N_1 < \cdots <N_m \leq n}{a_{N_m}\cdots a_{N_1}}
=(-1)^m\sum_{\sum{i.y_{k,i}}=m}{\prod_{i=1}^{m}{\frac{(-1)^{y_{k,i}}}{(y_{k,i})!} \left( \frac{1}{i}\sum_{N=1}^{n}{(a_N)^i }\right)^{y_{k,i}}}}.$$
\end{corollary}
\begin{corollary} \label{Corollary 3.2}
Let $m,n \in \mathbb{N}$, we have that 
$$(-1)^m\sum_{\sum{i.y_{k,i}}=m}{\prod_{i=1}^{m}{\frac{(-1)^{y_{k,i}}}{(y_{k,i})!} \left( \frac{n}{i}\right)^{y_{k,i}}}}=\binom{n}{m}.$$
\end{corollary}
\begin{proof}
From paper \cite{RepeatedSums}, we have the following relation, 
$$
\sum_{N_m=0}^{n-m}{\cdots \sum_{N_1=0}^{N_2}{1}}=\binom{n}{m}
.$$
We shift the variables in the following way,  
\begin{dmath*}
\sum_{q \leq N_1 < \cdots < N_m \leq n}{1}
=\sum_{N_m=m}^{n}{\sum_{N_{m-1}=m-1}^{N_m-1}\cdots \sum_{N_1=1}^{N_2-1}{1}}
=\sum_{N_m=0}^{n-m}{\sum_{N_{m-1}=m-1}^{N_m+m-1}{\cdots \sum_{N_1=1}^{N_2-1}{1}}}
=\sum_{N_m=0}^{n-m}{\sum_{N_{m-1}=0}^{N_m}{\sum_{N_{m-2}=m-2}^{N_{m-1}+m-2}{\cdots \sum_{N_1=1}^{N_2-1}{1}}}}
.\end{dmath*}
After completing the shifting for all variables, we get 
$$
\sum_{q \leq N_1 < \cdots < N_m \leq n}{1}
=\sum_{N_m=0}^{n-m}{\cdots \sum_{N_1=0}^{N_2}{1}}
=\binom{n}{m}
.$$
By applying Theorem \ref{Theorem 3.1}, we get 
$$
(-1)^m\sum_{\sum{i.y_{k,i}}=m}{\prod_{i=1}^{m}{\frac{(-1)^{y_{k,i}}}{(y_{k,i})!} \left( \frac{1}{i}\sum_{N=1}^{n}{1}\right)^{y_{k,i}}}}
=(-1)^m\sum_{\sum{i.y_{k,i}}=m}{\prod_{i=1}^{m}{\frac{(-1)^{y_{k,i}}}{(y_{k,i})!} \left( \frac{n}{i}\right)^{y_{k,i}}}}
=\binom{n}{m}
.$$
\end{proof}
\begin{example}
{\em For $n=1$, we get Lemma \ref{Lemma 3.1}, }
$$
\sum_{\sum{i.y_{k,i}}=m}{\prod_{i=1}^{m}{\frac{(-1)^{y_{k,i}}}{(y_{k,i})!i^{y_{k,i}}}}}
=(-1)^m\binom{1}{m}
=
\begin{cases}
(-1)^m & $ {\em for } $ 0 \leq m \leq 1, \\  
0 & $ {\em for } $ m \geq 2. \\ 
\end{cases}
$$
\end{example}
\begin{example}
{\em For $n=2$, }
$$
\sum_{\sum{i.y_{k,i}}=m}{\prod_{i=1}^{m}{\frac{(-2)^{y_{k,i}}}{(y_{k,i})!i^{y_{k,i}}}}}
=(-1)^m\binom{2}{m}
.$$
\end{example}
\begin{example}
{\em For $n=m$, }
$$
\sum_{\sum{i.y_{k,i}}=m}{\prod_{i=1}^{m}{\frac{(-m)^{y_{k,i}}}{(y_{k,i})!i^{y_{k,i}}}}}
=(-1)^m\binom{m}{m}=(-1)^m 
.$$
\end{example}
\begin{corollary} \label{Corollary 3.3}
For any $q,n \in \mathbb{N}$ where $n \geq q$, we have that 
$$
(-1)^{n-q+1}\sum_{\sum{i.y_{k,i}}=n-q+1}{\prod_{i=1}^{n-q+1}{\frac{(-1)^{y_{k,i}}}{(y_{k,i})!} \left( \frac{1}{i}\sum_{N=q}^{n}{(a_N)^i }\right)^{y_{k,i}}}}
=\prod_{j=q}^{n}{a_j}
.$$
\end{corollary}
\begin{proof}
$$
\sum_{q \leq N_1 < \cdots <N_{n-q+1} \leq n}{a_{N_{n-q+1}}\cdots a_{N_1}}
=\sum_{q = N_1, N_2=q+1, \cdots ,N_{n-q+1} = n}{a_{N_{n-q+1}}\cdots a_{N_1}}
=a_n \cdots a_q
.$$
By using Theorem \ref{Theorem 3.1}, we get the desired theorem. 
\end{proof}
\subsection{Particular cases}
In this section, we will apply the reduction formula for the cases of $m$ from $1$ to $4$. These cases were independently proven using two distinct methods (which are omitted here for simplicity). Similarly, these formulas were verified for a certain range of $n$ using a computer program which calculated the right expression as well as the left expression then checked that they were equal. 
\begin{itemize}
\item For $m=1$ 
$$
\sum_{1 \leq N_1 \leq n}{a_{N_1}}
=\sum_{N=1}^{n}{a_{N}}
.$$
\item For $m=2$ 
$$
\sum_{1 \leq N_1 < N_2 \leq n}{a_{N_2}a_{N_1}}
=\frac{1}{2}\left(\sum_{N=1}^{n}{a_{N}}\right)^{2}
-\frac{1}{2}\left(\sum_{N=1}^{n}{\left(a_{N}\right)^{2}}\right)
.$$
\item For $m=3$ 
$$
\sum_{1 \leq N_1 < N_2 < N_3 \leq n}{a_{N_3}a_{N_2}a_{N_1}}
=\frac{1}{6}\left(\sum_{N=1}^{n}{a_{N}}\right)^{3}
-\frac{1}{2}\left(\sum_{N=1}^{n}{a_{N}}\right)\left(\sum_{N=1}^{n}{\left(a_{N}\right)^{2}}\right)
+\frac{1}{3}\left(\sum_{N=1}^{n}{\left(a_{N}\right)^{3}}\right)
.$$
\item For $m=4$ 
\begin{dmath*}
\sum_{1 \leq N_1 < N_2 < N_3 < N_4 \leq n}{a_{N_4}a_{N_3}a_{N_2}a_{N_1}}
=\frac{1}{24}\left(\sum_{N=1}^{n}{a_{N}}\right)^{4}
-\frac{1}{4}\left(\sum_{N=1}^{n}{a_{N}}\right)^{2}\left(\sum_{N=1}^{n}{\left(a_{N}\right)^{2}}\right)
+\frac{1}{3}\left(\sum_{N=1}^{n}{a_{N}}\right)\left(\sum_{N=1}^{n}{\left(a_{N}\right)^{3}}\right)
+\frac{1}{8}\left(\sum_{N=1}^{n}{\left(a_{N}\right)^{2}}\right)^{2}
-\frac{1}{4}\left(\sum_{N=1}^{n}{\left(a_{N}\right)^{4}}\right)
.\end{dmath*}
\end{itemize}
\subsection{General Reduction Theorem}
Let $|A|$ represent the number of elements in a set $A$. Note that if $A$ is a set of sets then $|A|$ represents the number of sets in $A$. \\ 
Let $m$ be a non-negative integer and let $\{(y_{1,1}, \cdots, y_{1,m}),(y_{2,1}, \cdots, y_{2,m}), \cdots \}$ be the set of all partitions of $m$. 
Let us consider the set $M=\{1, \ldots, m\}$: 
The permutation group $S_m$ is the set of all permutations of the set $M$. Let $\sigma \in S_m$ be a permutation of $M$ and let $\sigma (i)$ represent the $i$-th element of this given permutation. The number of such permutations is given by 
\begin{equation}
|S_m|=m!
.\end{equation}
The cycle-type of a permutation $\sigma$ is the ordered set where the $i$-th element represents the number of cycles of size $i$ in the cycle decomposition of $\sigma$. The number of ways of arranging $i$ elements cyclically is $(i-1)!$. The number of possible combinations of $y_{k,i}$ cycles of size $i$ is $[(i-1)!]^{y_{k,i}}$. Hence, the number of permutations having cycle-type $(y_{k,1}, \cdots, y_{k,m})$ is given by 
\begin{equation}
\prod_{i=1}^{m}{[(i-1)!]^{y_{k,i}}}
.\end{equation}
A partition $P$ of a set $M$ is a set of non-empty disjoint subsets of $M$ such that every element of $M$ is present in exactly one of the subsets. Let $P=\{\underbrace{P_{1,1}, \cdots, P_{1,y_1}}_{y_1 \,\, sets}, \cdots, \underbrace{P_{m,1}, \cdots, P_{m,y_m}}_{y_m \,\, sets}\}$ represent a partition of a set of $m$ elements (for our purpose let it be the set $\{1, \ldots, m\}$). $P_{i,y}$ represents the $y$-th subset of order (size) $i$. $y_i$ represents the number of subsets of size $i$ contained in this partition of the set. It is interesting to note that $(y_1, \cdots, y_m)$ will always form a partition of $m$. 
However, the number of partitions of $m$ is smaller than the number of partitions of the set $M$ because there are more than one partition of the set $M$ that can be associated with a given partition of $m$. In fact, we can determine that the number of partitions of the set $M$ associated with the partition $(y_1, \cdots, y_m)$ is given by 
\begin{equation} \label{setpartition}
|\Omega_k|
=\frac{m!}{1!^{y_{k,1}}\cdots m!^{y_{k,m}}(y_{k,1})! \cdots (y_{k,m})!}
=\frac{m!}{\prod_{i=1}^{m}{i!^{y_{k,i}}y_{k,i}!}}
.\end{equation}  
where $\Omega_k$ is the set of all partitions of the set $M$ associated the partition $(y_{k,i}, \cdots, y_{k,m})$. This is because the number of ways to divide $m$ objects into $l_1$ groups of $1$ element, $l_2$ groups of $2$ elements, $\cdots$, and $l_m$ groups of $m$ elements is given by 
\begin{equation} 
\frac{m!}{1!^{l_1} \cdots m!^{l_{m}}l_1! \cdots l_m!}
=\frac{m!}{\prod_{i=1}^{m}{i!^{l_{i}}l_i!}}
.\end{equation}  
We will denote by $\Omega$ the set of all partitions of the set $M$. 
Finally, a partition $P$ of a set $M$ is a refinement of a partition $\rho$ of the same set $M$ if every element in $P$ is a subset of an element in $\rho$. We denote this as $P \succeq \rho$. \\
Using the notation introduced, we can formulate a generalization of Theorem \ref{Theorem 3.1} where all sequences are distinct. 
\begin{theorem} \label{Theorem 3.2}
Let $m,n,q \in \mathbb{N}$ such that $n \geq q+m-1$. Let $a_{(1);N}, \ldots , a_{(m);N}$ be $m$ sequences defined in the interval $[q,n]$.
we have that 
\begin{equation*}
\begin{split}
&\sum_{\sigma \in S_m}{\left(\sum_{q \leq N_1 < \cdots < N_m \leq n}{a_{(\sigma(m));N_{m}} \cdots a_{(\sigma(1));N_{1}} }\right)} \\
&=\sum_{\substack{P \in \Omega}}{(-1)^{m-\sum{y_{k,i}}}\prod_{i=1}^{m}{[(i-1)!]^{y_{k,i}} \left[\prod_{g=1}^{y_{k,i}}{\left( \sum_{N=q}^{n}{\prod_{h \in P_{i,g}}{a_{(h);N}}}\right)}\right]}}
.\end{split}
\end{equation*}
\end{theorem}
\begin{remark}
{\em The theorem can also be written as }
\begin{equation*}
\begin{split}
&\sum_{\sigma \in S_m}{\left(\sum_{q \leq N_1 < \cdots < N_m \leq n}{a_{(\sigma(m));N_{m}} \cdots a_{(\sigma(1));N_{1}} }\right)} \\
&\,\,=\sum_{\substack{k \\ \sum{i.y_{k,i}}=m}}{(-1)^{m-\sum{y_{k,i}}}\sum_{\Omega_{k}}{\prod_{i=1}^{m}{[(i-1)!]^{y_{k,i}} \left[\prod_{g=1}^{y_{k,i}}{\left( \sum_{N=q}^{n}{\prod_{h \in P_{i,g}}{a_{(h);N}}}\right)}\right]}}} \\
&\,\,=(-1)^m |S_m|\sum_{\substack{k \\ \sum{i.y_{k,i}}=m}}{\frac{1}{|\Omega_k|}\sum_{\Omega_{k}}{\prod_{i=1}^{m}{\frac{(-1)^{y_{k,i}}}{y_{k,i}! i^{y_{k,i}}} \left[\prod_{g=1}^{y_{k,i}}{\left( \sum_{N=q}^{n}{\prod_{h\in P_{i,g}}{a_{(h);N}}}\right)}\right]}}}.
\end{split}
\end{equation*}
{\em The first form is obtained by regrouping together, from the set of all partitions of the set $\{1, \cdots, m\}$, those who are associated with a given partition of $m$. \\
The second expression is obtained by noting that $$\frac{(-1)^{m}|S_m|}{|\Omega_k|}\prod_{i=1}^{m}{\frac{(-1)^{y_{k,i}}}{y_{k,i}!i^{y_{k,i}}}}=(-1)^{m-\sum{y_{k,i}}}\prod_{i=1}^{m}{{[(i-1)!]^{y_{k,i}}}}.$$
These forms are shown as they can be more easily used to show that this theorem reduces to Theorem \ref{Theorem 3.1} if all sequences are the same. } 
\end{remark}
\begin{proof}
The same terms appear in both sides of the theorem, hence, to prove the theorem, it suffices to prove that they appear with the same multiplicity on both sides. \\
We assume, without lost of generality, that all sequences are distinct. The left hand side term can be written as follows 
\begin{equation*}
\begin{split}
\sum_{\sigma \in S_m}{\left(\sum_{q \leq N_1 < \cdots < N_m \leq n}{a_{(\sigma(m));N_{m}} \cdots a_{(\sigma(1));N_{1}} }\right)} 
=\sum_{\sigma \in S_m}{\left(\sum_{q \leq N_1 < \cdots < N_m \leq n}{a_{(m);N_{\sigma(m)}} \cdots a_{(1);N_{\sigma(1)}} }\right)}
.\end{split}
\end{equation*}
Consider the symmetric group $S_m$ acting on $N=(N_1, \cdots, N_m)$. $N$ has an isotropy group $S_m(N)$ and an associated partition $\rho$ of the set $\{1, \cdots, m\}$. $\rho$ is the set of equivalence classes of the relation given by $a \sim b$ if and only if $N_a=N_b$. $S_m(N)=\{\sigma \in S_m \,\, | \,\, \sigma(i) \sim i\}$. Hence, a term 
\begin{equation} \label{seq}
a_{(m);N_m}\cdots a_{(1);N_1}
\end{equation}
appears, in the left hand side term, once if all the $N_j$'s are distinct and and none otherwise. \\
A term \eqref{seq} appears 
\begin{equation}  \label{rho}
\sum_{P \succeq \rho}{(-1)^{m-\sum{y_{k,i}}}\prod_{i=1}^{m}{{[(i-1)!]^{y_{k,i}}}}}
\end{equation}
times in the right hand side. Note that $\sum{y_{k,i}}$ is equal to the number of sets in $P$. \\
To prove the theorem, one has to show that 
\begin{equation} 
\sum_{P \succeq \rho}{(-1)^{m-\sum{y_{k,i}}}\prod_{i=1}^{m}{{[(i-1)!]^{y_{k,i}}}}}
=
\begin{cases}
1 &  \text{if } |\rho|=m,\\
0 &  \text{otherwise}.
\end{cases}
\end{equation}
We notice that the sign of $(-1)^{m-\sum{y_{k,i}}}\prod_{i=1}^{m}{{[(i-1)!]^{y_{k,i}}}}$ is positive if the permutations of cycle type $P$ are even and negative if they are odd. 
Therefore, \eqref{rho} is the signed sum of the number of even and odd permutations in the isotropy group $S_m(N)$. Let us also note that an isotropy group has the same number of even and odd permutations unless the associated partition $\rho$ is $\{\{1\}, \cdots, \{m\}\}$ ($|\rho|=m$). Hence, \eqref{rho} is zero unless $|\rho|=m$. This concludes our proof of the theorem.   
\end{proof} 
\begin{example}
{\em For $m=2$, Theorem \ref{Theorem 3.2} gives the following, }
\begin{dmath*}
\sum_{q \leq N_1 < N_2 \leq n}{a_{N_{2}}  b_{N_{1}} }
+\sum_{q \leq N_1 < N_2 \leq n}{b_{N_{2}}  a_{N_{1}} }
=\left(\sum_{N=q}^{n}{a_N}\right)\left(\sum_{N=q}^{n}{b_N}\right)
-\left(\sum_{N=q}^{n}{a_N b_N}\right)
.\end{dmath*}
\end{example}
\begin{example}
{\em For $m=3$, Theorem \ref{Theorem 3.2} gives the following, }
\begin{dmath*}
\sum_{\sigma \in S_3}{\left(\sum_{q \leq N_1 < N_2 < N_3 \leq n}{a_{(\sigma(3));N_{3}} a_{(\sigma(2));N_{2}}  a_{(\sigma(1));N_{1}} }\right)}
=\left(\sum_{N=q}^{n}{a_{(1);N}}\right)\left(\sum_{N=q}^{n}{a_{(2);N}}\right)\left(\sum_{N=q}^{n}{a_{(3);N}}\right)
-\left(\sum_{N=q}^{n}{a_{(1);N}}\right)\left(\sum_{N=q}^{n}{a_{(2);N} a_{(3);N}}\right)
-\left(\sum_{N=q}^{n}{a_{(2);N}}\right)\left(\sum_{N=q}^{n}{a_{(1);N} a_{(3);N}}\right)
-\left(\sum_{N=q}^{n}{a_{(3);N}}\right)\left(\sum_{N=q}^{n}{a_{(1);N} a_{(2);N}}\right)
+2\left(\sum_{N=q}^{n}{a_{(1);N} a_{(2);N} a_{(3);N}}\right)
.\end{dmath*}
\end{example}
\section{Applications to polynomials} \label{ApplicationstoPolynomials}
Multiple sums have a variety of applications. However, the most famous one is it's usage in Vi\`ete's formula to relate the coefficients of a polynomial to its roots. This corresponds to the particular case where the sequence $a_N$ represents the roots $r_N$ of the polynomial. In this section, some applications of this type of sums to polynomials will be presented. 
\subsection{Relation between the roots and the coefficients of a polynomial}
Let $P(x)$ be a polynomial of degree $n$, 
$$
P(x)=\sum_{i=0}^{n}{a_i x^i}
=a_n x^n + a_{n-1} x^{n-1} + \cdots + a_1 x +a_0
$$
where $a_i$ is the coefficient of $x^i$. \\
Vi\`ete's theorem allows us to rewrite this polynomial as a product of its factors, 
$$
P(x)=a_n \prod_{i=1}^{n}{(x-r_i)}=a_n (x-r_1) \cdots (x-r_n)
$$
where $r_i$ is the $i$-th root of the polynomial. \\ 
The relation between the roots and the coefficients of a polynomial of degree $n$ (Vi\`ete's formula) is structured as a multiple sum of the roots of the polynomial, 
$$
\frac{a_m}{a_n}=(-1)^{n-m} \sum_{1 \leq N_1 < \cdots < N_{n-m} \leq n}{r_{N_{n-m}} \cdots r_{N_1}}
.$$
This relation can also be written as follows, 
$$
\frac{a_{n-m}}{a_n}=(-1)^{m} \sum_{1 \leq N_1 < \cdots < N_{m} \leq n}{r_{N_{m}} \cdots r_{N_1}}
.$$
This relation linking the roots and coefficients of a polynomial can be simplified by applying the reduction theorem to Vi\`ete's formula.  
\begin{theorem} \label{Theorem 4.1}
Let $P(x)=a_n x^n + \cdots + a_1 x +a_0 =a_n(x^n + \cdots +\frac{a_1}{a_n}x+\frac{a_0}{a_n})$ be a polynomial and let $r_1, \cdots , r_n$ be the roots of this polynomial. The coefficients of this polynomial can be linked to its roots by the following relation, 
$$
\frac{a_{n-m}}{a_n}
=\sum_{\sum{i.y_{k,i}}=m}{\prod_{i=1}^{m}{\frac{(-1)^{y_{k,i}}}{(y_{k,i})!} \left( \frac{1}{i}\sum_{N=1}^{n}{(r_N)^i }\right)^{y_{k,i}}}}
.$$
\end{theorem}
\begin{proof}
Vi\`ete's formula is as follows, 
$$
\frac{a_{n-m}}{a_n}=(-1)^{m} \sum_{1 \leq N_1 < \cdots < N_{m} \leq n}{r_{N_{m}} \cdots r_{N_1}}
.$$
Hence, by applying Corollary \ref{Corollary 3.1}, we get the desired theorem. 
\end{proof}
Some relations relating the coefficients of a polynomial to its roots for different values of $m$ have been calculated using Theorem \ref{Theorem 4.1}, and are as follows: 
\begin{itemize}
\item For $m=0$
$$
\frac{a_n}{a_n}=1
.$$
\item For $m=1$
$$
\frac{a_{n-1}}{a_n}=-\sum_{N=1}^{n}{r_{N}}
.$$
\item For $m=2$
$$
\frac{a_{n-2}}{a_n}
=\frac{1}{2}\left(\sum_{N=1}^{n}{r_{N}}\right)^{2}
-\frac{1}{2}\left(\sum_{N=1}^{n}{\left(r_{N}\right)^{2}}\right)
.$$
\item For $m=3$
$$
\frac{a_{n-3}}{a_n}
=-\frac{1}{6}\left(\sum_{N=1}^{n}{r_{N}}\right)^{3}
+\frac{1}{2}\left(\sum_{N=1}^{n}{r_{N}}\right)\left(\sum_{N=1}^{n}{\left(r_{N}\right)^{2}}\right)
-\frac{1}{3}\left(\sum_{N=1}^{n}{\left(r_{N}\right)^{3}}\right)
.$$
\item For $m=4$
\begin{dmath*}
\frac{a_{n-4}}{a_n}
=\frac{1}{24}\left(\sum_{N=1}^{n}{r_{N}}\right)^{4}
-\frac{1}{4}\left(\sum_{N=1}^{n}{r_{N}}\right)^{2}\left(\sum_{N=1}^{n}{\left(r_{N}\right)^{2}}\right)
+\frac{1}{3}\left(\sum_{N=1}^{n}{r_{N}}\right)\left(\sum_{N=1}^{n}{\left(r_{N}\right)^{3}}\right)
+\frac{1}{8}\left(\sum_{N=1}^{n}{\left(r_{N}\right)^{2}}\right)^{2}
-\frac{1}{4}\left(\sum_{N=1}^{n}{\left(r_{N}\right)^{4}}\right)
.\end{dmath*}
\item For $m=n$
$$
\frac{a_0}{a_n}
=(-1)^n \prod_{i=1}^{n}{r_i}
.$$
\end{itemize}
\subsection{Relation between the roots of a polynomial and the roots of its derivatives}
With the formulas developed, we can go beyond just linking the coefficients and roots of a polynomial. In this section, we will develop a formula linking the roots of a polynomial to the roots of its derivatives as illustrated by the following theorem.  
\begin{theorem}
Let $f(x)=\sum_{i=0}^{n}{a_ix^i}=a_n(x-r_1) \cdots (x-r_n)$ be a polynomial of order $n$ and let 
$$
f^{(k)}(x)=\sum_{i=0}^{n-k}{a_{(k);i}x^i}
=\sum_{j=k}^{n}{\frac{j!}{(j-k)!}a_{j}x^{j-k}}
=a_n(x-r_{(k);1}) \cdots (x-r_{(k);n-k})
$$
be its derivative of order $k$. The roots of $f(x)$ can be linked to the roots of $f^{(k)}(x)$ by the following equation, 
\begin{dmath*}
\frac{(n-m)!}{n!}\sum_{\sum{i.y_{k,i}}=m}{\prod_{i=1}^{m}{\frac{(-1)^{y_{k,i}}}{(y_{k,i})!} \left( \frac{1}{i}\sum_{N=1}^{n}{(r_N)^i }\right)^{y_{k,i}}}}
=\frac{(n-m-k)!}{(n-k)!}\sum_{\sum{i.y_{k,i}}=m}{\prod_{i=1}^{m}{\frac{(-1)^{y_{k,i}}}{(y_{k,i})!} \left( \frac{1}{i}\sum_{N=1}^{n-k}{(r_{(k);N})^i }\right)^{y_{k,i}}}}
.\end{dmath*}
\end{theorem}  
\begin{proof}
By comparing the coefficients of $f(x)$ and $f^{(k)}(x)$, we get 
$$
\frac{a_{(k);n-k-m}}{a_{(k);n-k}}
=\frac{\frac{(n-m)!}{(n-m-k)!}a_{n-m}}{\frac{n!}{(n-k)!}a_{n}}
\frac{\frac{(n-m)!}{n!}}{\frac{(n-m-k)!}{(n-k)!}} \frac{a_{n-m}}{a_n}
.$$
By applying Vi\`ete's formula, we get the desired theorem.  
\end{proof}
A special case of this theorem which is of special interest is the following.  
\begin{corollary} \label{Corollary 4.1}
Let $f(x)=\sum_{i=0}^{n}{a_ix^i}=a_n(x-r_1) \cdots (x-r_n)$ be a polynomial of order $n$ and let 
$$
f^{(k)}(x)=\sum_{i=0}^{n-k}{a_{(k);i}x^i}
=\sum_{j=k}^{n}{\frac{j!}{(j-k)!}a_{j}x^{j-k}}
=a_n(x-r_{(k);1}) \cdots (x-r_{(k);n-k})
$$
be its derivative of order $k$. Let $\overline{x}$ be the average root value for $f(x)$ and $\overline{x_{(k)}}$ be the average root value for $f^{(k)}(x)$. 
$$
\overline{x}
=\frac{r_1 + \cdots + r_n}{n} 
=\frac{r_{(k);1} + \cdots + r_{(k);n-k}}{n-k} 
=\overline{x_{(k)}}
.$$
\end{corollary}
\subsection{Generalized Binomial Theorem}
In this section, we prove a formula for the sum of multiple sums. This identity is then used to prove a generalization of the binomial theorem as well as a few MZV identities. 
\begin{theorem} \label{general sum of multiple sums}
Let $x\in \mathbb{C^*}$ and $n\in \mathbb{N^*}$. We have that 
$$
\sum_{m=0}^{n}{(-1)^m x^{n-m}\sum_{1\leq N_1 < \cdots < N_m \leq n}{r_{N_m}\cdots r_{N_1}}}=(-1)^n(r_1-x)\cdots(r_n-x)
$$
{\em and} 
$$
\sum_{m=0}^{n}{x^{n-m}\sum_{\sum{iy_i}=m}{\prod_{i=1}^{m}{\frac{(-1)^{y_i}}{y_i!}\left(\frac{1}{i}\sum_{N=1}^{n}{(r_N)^i}\right)^{y_i}}}}=(-1)^n(r_1-x)\cdots(r_n-x).
$$
\end{theorem}
\begin{proof}
Let $f(x)=\sum_{m=0}^{n}{a_mx^m}=a_n(x-r_1) \cdots (x-r_n)$ be a polynomial of order $n$. Then 
$$\frac{f(x)}{a_n}
=\sum_{m=0}^{n}{x^m\frac{a_m}{a_n}}
=\sum_{m=0}^{n}{x^{n-m}\frac{a_{n-m}}{a_n}}
=(x-r_1) \cdots (x-r_n)
=(-1)^n(r_1-x) \cdots (r_n-x).$$ 
Applying Vi\`ete's formula, we get the first identity of the theorem. Similarly, applying Theorem \ref{Theorem 4.1} instead leads to the second identity. 
\end{proof}
\begin{corollary} \label{alternating sum of multiple sums}
For any $n\in \mathbb{N^*}$, we have that 
$$
\sum_{m=0}^{n}{(-1)^m\sum_{1\leq N_1 < \cdots < N_m \leq n}{r_{N_m}\cdots r_{N_1}}}=(-1)^n(r_1-1)\cdots(r_n-1)
$$
{\em and} 
$$
\sum_{m=0}^{n}{\sum_{\sum{iy_i}=m}{\prod_{i=1}^{m}{\frac{(-1)^{y_i}}{y_i!}\left(\frac{1}{i}\sum_{N=1}^{n}{(r_N)^i}\right)^{y_i}}}}=(-1)^n(r_1-1)\cdots(r_n-1).
$$
\end{corollary}
\begin{corollary} \label{sum of multiple sums}
Substituting $r_N$ by $(-r_N)$ in Corollary \ref{alternating sum of multiple sums}, we get the sum of multiple sums.
$$
\sum_{m=0}^{n}{\sum_{1\leq N_1 < \cdots < N_m \leq n}{r_{N_m}\cdots r_{N_1}}}=\prod_{N=1}^{n}{(r_N+1)}.
$$
\end{corollary}
\begin{theorem} \label{Generalized BT}
For any $n\in \mathbb{N^*}$, we have that 
\begin{equation*}
\begin{split}
(a_1+b_1)\cdots(a_n+b_n)
&=\left(\prod_{i=1}^{n}{b_i}\right)\sum_{m=0}^{n}{\sum_{1\leq N_1 < \cdots < N_m \leq n}{\frac{a_{N_m}\cdots a_{N_1}}{b_{N_m}\cdots b_{N_1}}}} \\
&=\left(\prod_{i=1}^{n}{b_i}\right)\sum_{m=0}^{n}{(-1)^m\sum_{\sum{iy_i}=m}{\prod_{i=1}^{m}{\frac{(-1)^{y_i}}{y_i!}\left(\frac{1}{i}\sum_{N=1}^{n}{\left(\frac{a_N}{b_N}\right)^i}\right)^{y_i}}}}. 
\end{split}
\end{equation*}
\end{theorem}
\begin{proof} 
$$
\prod_{i=1}^{n}{(a_i+b_i)}
=\prod_{i=1}^{n}{(b_i)\left(\frac{a_i}{b_i}+1\right)}
=\left(\prod_{i=1}^{n}{b_i}\right)\prod_{i=1}^{n}{\left(\frac{a_i}{b_i}+1\right)}.
$$
Letting $r_i=(a_i/b_i)$ and applying Corollary \ref{sum of multiple sums}, we obtain this theorem. 
\end{proof} 
Additionally, Corollary \ref{sum of multiple sums} can be used to derive several identities related to MZVs. 
\begin{corollary} \label{zeta sum}
The sum of MZVs can be rewritten as a product as follows: 
\begin{equation*} 
\sum_{m=0}^{n}{\zeta_n(\underbrace{p, \ldots, p}_{m \,\,times})}=\prod_{N=1}^{n}{\left(1+\frac{1}{N^p}\right)} 
\,\,\,\,\,\text{and}\,\,\,\,\, 
\sum_{m=0}^{\infty}{\zeta(\underbrace{p, \ldots, p}_{m \,\,times})}=\prod_{N=1}^{\infty}{\left(1+\frac{1}{N^p}\right)} 
.\end{equation*}
\end{corollary}
\begin{proof}
Applying Corollary \ref{sum of multiple sums} with  $r_N=(1/N^{p})$, we find this identity.
\end{proof}
\begin{corollary} \label{limit of zeta sum}
The sum of MZVs converges to 2 as the argument of the MZVs goes to infinity. 
\begin{equation*} 
\lim_{p \to \infty}{\sum_{m=0}^{n}{\zeta_n(\underbrace{p, \ldots, p}_{m \,\,times})}}
=\lim_{p \to \infty}{\sum_{m=0}^{\infty}{\zeta(\underbrace{p, \ldots, p}_{m \,\,times})}}
=2
.\end{equation*}
\end{corollary}
\begin{proof}
\begin{equation*} 
\lim_{p \to \infty}{\sum_{m=0}^{n}{\zeta_n(\underbrace{p, \ldots, p}_{m \,\,times})}}=\lim_{p \to \infty}{\prod_{N=1}^{n}{\left(1+\frac{1}{N^p}\right)}} 
=\lim_{p \to \infty}{\left(1+1\right)\prod_{N=2}^{n}{\left(1+\frac{1}{N^p}\right)}} 
=(2)(1)=2
.\end{equation*}
Letting $n \to \infty$, we obtain the second part of the identity. 
\end{proof}
Another identity that can be found through Corollary \ref{sum of multiple sums} is the following. 
\begin{corollary} \label{sum of the multiple sums of N}
For any $n\in \mathbb{N^*}$, we have that 
$$
\sum_{m=0}^{n}{\sum_{1\leq N_1 < \cdots < N_m \leq n}{N_m\cdots N_1}}=\prod_{N=1}^{n}{(N+1)}=(n+1)!.
$$
\end{corollary}
\section{Other applications} \label{Applications}
In this section, we will apply the reduction formula presented in Theorem \ref{Theorem 3.1} to simplify certain special multiple sums. The first special sum that we will simplify is the multiple sum of $N^p$. The second special sum is the multiple harmonic series as well as the multiple $p$-series for even values of $p$. 
\subsection{Multiple Power Sum}
The Faulhaber formula is a formula developed by Faulhaber in a 1631 edition of Academia Algebrae \cite{Faulhaber} to calculate sums of powers. The Faulhaber formula is as follows 
$$
\sum_{N=1}^{n}{N^p}=\frac{1}{p+1}\sum_{j=0}^{p}{(-1)^j\binom{p+1}{j}B_{j}n^{p+1-j}}
$$ 
where $B_j$ are the Bernoulli numbers of the first kind.
\begin{remark}
{\em See \cite{nielsen1923traite} for details on the history of Bernoulli numbers. }
\end{remark}

In this section, we will use the reduction formula for multiple sums to develop a more general form of the Faulhaber formula. 
\begin{theorem} \label{Theorem 5.1}
For any $m,n,p \in \mathbb{N}$ such that $n \geq m$, we have that 
\begin{equation*}
\begin{split}
\sum_{1 \leq N_1 < \cdots <N_m \leq n}{{N_m}^p\cdots {N_1}^p}
&=(-1)^m\sum_{\sum{i.y_{k,i}}=m}{\prod_{i=1}^{m}{\frac{(-1)^{y_{k,i}}}{(y_{k,i})!i^{y_{k,i}}} \left( \sum_{N=1}^{n}{N^{ip}}\right)^{y_{k,i}}}} \\
&=(-1)^m\sum_{\sum{i.y_{k,i}}=m}{\prod_{i=1}^{m}{\frac{(-1)^{y_{k,i}}}{(y_{k,i})!i^{y_{k,i}}} \left(\frac{n^{ip+1}}{ip+1}\sum_{j=0}^{ip}{(-1)^j\binom{ip+1}{j}\frac{B_{j}}{n^{j}}} \right)^{y_{k,i}}}} \\
\end{split}
\end{equation*}
where $B_j$ are the Bernoulli numbers of the first kind. 
\end{theorem}
\begin{proof}
This theorem is obtained by applying Theorem \ref{Theorem 3.1} and then applying Faulhaber's formula. 
\end{proof}
\begin{corollary} \label{Corollary 5.1}
For any $m,n \in \mathbb{N}$, we have that 
\begin{equation*}
\begin{split}
{n+1 \brack n-m+1}
&=(-1)^m\sum_{\sum{i.y_{k,i}}=m}{\prod_{i=1}^{m}{\frac{(-1)^{y_{k,i}}}{(y_{k,i})!i^{y_{k,i}}} \left( \sum_{N=1}^{n}{N^{i}}\right)^{y_{k,i}}}} \\
&=(-1)^m\sum_{\sum{i.y_{k,i}}=m}{\prod_{i=1}^{m}{\frac{(-1)^{y_{k,i}}}{(y_{k,i})!i^{y_{k,i}}} \left(\frac{n^{i+1}}{i+1}\sum_{j=0}^{i}{(-1)^j\binom{i+1}{j}\frac{B_{j}}{n^{j}}} \right)^{y_{k,i}}}}
\end{split}
\end{equation*}
where $B_j$ are the Bernoulli numbers of the first kind.
\end{corollary}
\begin{proof}
Knowing that Stirling numbers are the coefficients of a polynomial with roots $0,1, \cdots, n-1$, then, from Vi\`ete's formula, we express a Stirling number of the first kind as follows, 
$$
\sum_{0 \leq N_1 < \cdots <N_m \leq n-1}{{N_m}\cdots {N_1}}
={n \brack n-m}
.$$
Knowing that terms containing $N_1=0$ will be zero and substituting $n$ for $n-1$, we get 
$$
\sum_{1 \leq N_1 < \cdots <N_m \leq n}{{N_m}\cdots {N_1}}
={n+1 \brack n-m+1}
.$$
By applying Theorem \ref{Theorem 5.1} for $p=1$, the theorem is proven. 
\end{proof}
Let us consider the following special cases: 
\begin{itemize}
\item For $m=2$
\begin{equation*}
\begin{split}
&\sum_{1 \leq N_1 < N_2 \leq n}{{N_2^p}{N_1^p}} \\
&\,\,=\frac{1}{2}\left(\sum_{N=1}^{n}{{N}^p}\right)^{2}
-\frac{1}{2}\left(\sum_{N=1}^{n}{{N}^{2p}}\right) \\
&\,\,=\frac{1}{2}\left[\left(\frac{n^{p+1}}{p+1}\sum_{j=0}^{p}{(-1)^j\binom{p+1}{j}\frac{B_{j}}{n^{j}}} \right)^{2}
-\left(\frac{n^{2p+1}}{2p+1}\sum_{j=0}^{2p}{(-1)^j\binom{2p+1}{j}\frac{B_{j}}{n^{j}}} \right)\right]
.\end{split}
\end{equation*}
\begin{example}
{\em For $p=1$, by using the previous equation and exploiting Faulhaber's formulas, we can get the following formula }
$$
\sum_{1 \leq N_1 < N_2 \leq n}{{N_2}{N_1}}
=\frac{n(n-1)(n+1)(3n+2)}{24}
=\left(\sum_{N=1}^{n}{N}\right)\left[\frac{(n-1)(3n+2)}{12}\right]
.$$
\end{example}
\begin{example}
{\em For $p=2$, by applying this theorem and exploiting Faulhaber's formulas, we can get the following formula }
$$
\sum_{1 \leq N_1 < N_2 \leq n}{{N_2^2}{N_1^2}}
=\frac{n(n-1)(n+1)(2n-1)(2n+1)(5n+6)}{360}
.$$
\end{example}
\item For $m=3$ \\ 
$$
\sum_{1 \leq N_1<N_2<N_3 \leq n}{{N_3}^p{N_2}^p{N_1}^p}
=\frac{1}{6}\left(\sum_{N=1}^{n}{N^p}\right)^{3}
-\frac{1}{2}\left(\sum_{N=1}^{n}{N^p}\right)\left(\sum_{N=1}^{n}{N^{2p}}\right)
+\frac{1}{3}\left(\sum_{N=1}^{n}{N^{3p}}\right)
.$$
\begin{example}
{\em For $p=1$, by applying this theorem and exploiting Faulhaber's formulas, we can get the following formula }
$$
\sum_{1 \leq N_1<N_2<N_3 \leq n}{{N_3}{N_2}{N_1}}
=\frac{(n - 2) (n - 1) n^2 (n + 1)^2}{48}
=\left(\sum_{N=1}^{n}{N}\right)^2\left[\frac{(n-1)(n-2)}{12}\right]
.$$
\end{example}
\end{itemize}
\subsection{Multiple harmonic series}
In this section, using the formula developed by Euler and the reduction theorem (Theorem \ref{Theorem 3.1}), we will prove an expression which can be used to calculate multiple harmonic series for positive even values. Then we will present new identities based on solutions for some more general forms of the Basel problem.\\  

We start by applying Theorem \ref{Theorem 3.1} to the zeta function for positive even values to get an expression for the multiple series of $\frac{1}{N^{2p}}$. 
\begin{theorem} \label{Theorem 5.2}
For any $m,p \in \mathbb{N}$, we have that 
\begin{equation*}
\begin{split}
\sum_{1 \leq N_1< \cdots <N_{m}}{\frac{1}{N_{m}^{2p} \cdots N_{1}^{2p}}}
&=(-1)^{m}\sum_{\sum{i.y_{k,i}}=m}{\prod_{i=1}^{m}{\frac{(-1)^{y_{k,i}}}{(y_{k,i})!i^{y_{k,i}}} \left(\zeta(2ip)\right)^{y_{k,i}}}} \\
&=(-1)^{(p+1)m}(2\pi)^{2pm}\sum_{\sum{i.y_{k,i}}=m}{\prod_{i=1}^{m}{\frac{1}{(y_{k,i})!} \left(\frac{B_{2ip}}{(2i)(2ip)!}\right)^{y_{k,i}}}}
.\end{split}
\end{equation*}
\end{theorem}
\begin{proof}
From Theorem \ref{Theorem 3.1}, 
\begin{equation*}
\begin{split}
\sum_{1 \leq N_1< \cdots <N_{m}}{\frac{1}{N_{m}^{2p} \cdots N_{1}^{2p}}}
&=(-1)^{m}\sum_{\sum{i.y_{k,i}}=m}{\prod_{i=1}^{m}{\frac{(-1)^{y_{k,i}}}{(y_{k,i})!i^{y_{k,i}}} \left(\sum_{N=1}^{\infty}{\left(\frac{1}{N^{2p}}\right)^{i}}\right)^{y_{k,i}}}} \\
&=(-1)^{m}\sum_{\sum{i.y_{k,i}}=m}{\prod_{i=1}^{m}{\frac{(-1)^{y_{k,i}}}{(y_{k,i})!i^{y_{k,i}}} \left(\zeta(2ip)\right)^{y_{k,i}}}} 
.\end{split}
\end{equation*}
Euler proved that, for $m \geq 1$ (see \cite{arfken}),  
$$
\zeta(2m)=\frac{(-1)^{m+1}(2\pi)^{2m}}{2(2m)!}B_{2m}
.$$
Hence, 
\begin{equation*}
\begin{split}
\sum_{1 \leq N_1< \cdots <N_{m}}{\frac{1}{N_{m}^{2p} \cdots N_{1}^{2p}}}
&=(-1)^m\sum_{\sum{i.y_{k,i}}=m}{\prod_{i=1}^{m}{\frac{(-1)^{y_{k,i}}}{(y_{k,i})!i^{y_{k,i}}} \left((-1)^{ip+1}\frac{B_{2ip}(2\pi)^{2ip}}{2(2ip)!}\right)^{y_{k,i}}}} \\
&=(-1)^{(p+1)m}(2\pi)^{2pm}\sum_{\sum{i.y_{k,i}}=m}{\prod_{i=1}^{m}{\frac{1}{(y_{k,i})!} \left(\frac{B_{2ip}}{(2i)(2ip)!}\right)^{y_{k,i}}}}
.\end{split}
\end{equation*}
\end{proof}
\begin{corollary} \label{Corollary 5.2}
For $m=2$, Theorem \ref{Theorem 5.2} becomes 
$$
\sum_{1 \leq N_1 < N_2}{\frac{1}{N_1^{2p} N_2^{2p}}}
= \frac{1}{2}\left[(\zeta(2p))^2-\zeta(4p)\right]=\frac{(2\pi)^{4p}}{4(4p)!}\left[\binom{4p}{2p}\frac{(B_{2p})^2}{2}+B_{4p}\right]
.$$
\end{corollary}
\begin{proof}
By applying Theorem \ref{Theorem 5.2} with $m=2$, we get 
\begin{equation*}
\begin{split}
\sum_{1 \leq N_1 < N_2}{\frac{1}{N_1^{2p} N_2^{2p}}}
&= \frac{1}{2}\left[(\zeta(2p))^2-\zeta(4p)\right]
=\frac{1}{2}\left[\frac{(2\pi)^{4p}}{4[(2p)!]^2}(B_{2p})^2-\frac{(2\pi)^{4p}}{2(4p)!}B_{4p}\right]\\
&=\frac{(2\pi)^{4p}}{4}\left[\frac{(B_{2p})^2}{2(2p)!(2p)!}-\frac{B_{4p}}{(4p)!}\right]
=\frac{(2\pi)^{4p}}{4(4p)!}\left[\binom{4p}{2p}\frac{(B_{2p})^2}{2}+B_{4p}\right]
.\end{split}
\end{equation*}
\end{proof}
The following table summarizes some values of the zeta function for positive even arguments: 
$$
\zeta(2)=\frac{\pi^2}{6} \,\,\,\,\,\,
\zeta(4)=\frac{\pi^4}{90} \,\,\,\,\,\,
\zeta(6)=\frac{\pi^6}{945} \,\,\,\,\,\,
\zeta(8)=\frac{\pi^8}{9450} \,\,\,\,\,\,
\zeta(10)=\frac{\pi^{10}}{93555} \,\,\,\,\,\,
$$
$$
\zeta(12)=\frac{691\pi^{12}}{638512875} \,\,\,\,\,\,
\zeta(14)=\frac{2\pi^{14}}{18243225} \,\,\,\,\,\,
\zeta(16)=\frac{3617\pi^{16}}{325641566250}. \,\,\,\,\,\,
$$
By using the values in the above table as well as Theorem \ref{Theorem 5.2} and playing with different values, we can notice some identities.  
\begin{theorem} \label{Theorem 5.3}
For any $m \in \mathbb{N}$, we have that 
$$
\sum_{1 \leq N_1 < \cdots < N_m}{\frac{1}{N_m^2 \cdots N_1^2}}
=\sum_{\sum{i.y_{k,i}}=m}{\prod_{i=1}^{m}{\frac{(-1)^{y_{k,i}}}{(y_{k,i})!i^{y_{k,i}}} \left(\zeta(2i)\right)^{y_{k,i}}}}
=\frac{\pi^{2m}}{(2m+1)!}
$$
and 
$$
\sum_{\sum{i.y_{k,i}}=m}{\prod_{i=1}^{m}{\frac{1}{(y_{k,i})!} \left(\frac{B_{2i}}{(2i)(2i)!}\right)^{y_{k,i}}}}
=\frac{1}{2^{2m}(2m+1)!}
.$$
\end{theorem}
\begin{proof}
In \cite{hoffman1992multiple}, \cite{ohno2001multiple}, and \cite{Schneider}, the following relation was proven, 
$$
\sum_{1 \leq N_1 < \cdots < N_m}{\frac{1}{N_m^2 \cdots N_1^2}}
=\frac{\pi^{2m}}{(2m+1)!}
.$$
By applying Theorem \ref{Theorem 3.1}, we get 
$$
\sum_{\sum{i.y_{k,i}}=m}{\prod_{i=1}^{m}{\frac{(-1)^{y_{k,i}}}{(y_{k,i})!i^{y_{k,i}}} \left(\zeta(2i)\right)^{y_{k,i}}}}
=\frac{\pi^{2m}}{(2m+1)!}
.$$
Applying Theorem \ref{Theorem 5.2} with $p=1$, we obtain the second equation. 
\end{proof}
\begin{example}
{\em For $m=4$, we have }
\begin{dmath*}
\sum_{1 \leq N_1<N_2<N_3<N_4}{\frac{1}{N_4^2 N_3^2 N_2^2 N_1^2}}
=\frac{1}{24}\left(\sum_{N=1}^{\infty}{\frac{1}{N^2}}\right)^{4}
-\frac{1}{4}\left(\sum_{N=1}^{\infty}{\frac{1}{N^2}}\right)^{2}\left(\sum_{N=1}^{\infty}{\frac{1}{N^4}}\right)
+\frac{1}{3}\left(\sum_{N=1}^{\infty}{\frac{1}{N^2}}\right)\left(\sum_{N=1}^{\infty}{\frac{1}{N^6}}\right)
+\frac{1}{8}\left(\sum_{N=1}^{\infty}{\frac{1}{N^4}}\right)^{2}
-\frac{1}{4}\left(\sum_{N=1}^{\infty}{\frac{1}{N^8}}\right)
={\frac{\pi^8}{9!} 
(\approx 0.02614784782)}
.\end{dmath*}
\end{example}
Using Theorem \ref{Theorem 5.3}, we will prove that the multiple sum of $\frac{1}{N^p}$ will converge to 0 as the number of summations $m$ goes to infinity for any integer $p \geq 2$. 
\begin{theorem} \label{Theorem 5.4}
Let $p \in \mathbb{N}$ such that $p \geq 2$, for any $m \in \mathbb{N}$, we have that  
$$ 
\lim_{m \to \infty}{{\left(\sum_{1 \leq N_1 < \cdots < N_m}{\frac{1}{N_m^{p} \cdots N_1^{p}}}\right)}}
=0
.$$
\end{theorem}
\begin{proof}
Knowing that for any integer $p \geq 2, 0 \leq \frac{1}{N_i^{p}} \leq \frac{1}{N_i^{2}}$, therefore, $0 \leq \frac{1}{N_m^{p} \cdots N_1^{p}} \leq \frac{1}{N_m^{2} \cdots N_1^{2}}$, which then implies that 
$$
0 \leq 
\sum_{1 \leq N_1 < \cdots < N_m}{\frac{1}{N_m^{p} \cdots N_1^{p}}} 
\leq 
\sum_{1 \leq N_1 < \cdots < N_m}{\frac{1}{N_m^{2} \cdots N_1^{2}}}
.$$
By taking the limit as $m$ goes to infinity and applying Theorem \ref{Theorem 5.3}, we get 
$$
0 \leq 
\lim_{m \to \infty}{{\left(\sum_{1 \leq N_1 < \cdots < N_m}{\frac{1}{N_m^{p} \cdots N_1^{p}}}\right)}}
\leq 
\lim_{m \to \infty}{{\left(\frac{\pi^{2m}}{(2m+1)!}\right)}}=0
.$$
Hence, the theorem is proven. 
\end{proof}
\begin{theorem} \label{Conjecture 5.1}
For any $m \in \mathbb{N}$, we have that 
$$
\sum_{1 \leq N_1 < \cdots < N_m}{\frac{1}{N_m^4 \cdots N_1^4}}
=\sum_{\sum{i.y_{k,i}}=m}{\prod_{i=1}^{m}{\frac{(-1)^{y_{k,i}}}{(y_{k,i})!i^{y_{k,i}}} \left(\zeta(4i)\right)^{y_{k,i}}}}
=\frac{2(2^{2m})\pi^{4m}}{(4m+2)!}
=\frac{2(\sqrt{2} \pi)^{4m}}{(4m+2)!}
$$
and 
$$
\sum_{\sum{i.y_{k,i}}=m}{\prod_{i=1}^{m}{\frac{1}{(y_{k,i})!} \left(\frac{B_{4i}}{(2i)(4i)!}\right)^{y_{k,i}}}}
=\frac{2(-1)^m}{2^{2m}(4m+2)!}
.$$
\end{theorem}
\begin{proof}
From \cite{borwein1996evaluations}, we have that 
$$
\sum_{1 \leq N_1 < \cdots < N_m}{\frac{1}{N_m^4 \cdots N_1^4}}
=\frac{2(2^{2m})\pi^{4m}}{(4m+2)!}
=\frac{2(\sqrt{2} \pi)^{4m}}{(4m+2)!}
.$$
Hence, applying Theorem \ref{Theorem 3.1}, we obtain the first equation of this theorem. \\
The second equation is obtained by applying Theorem \ref{Theorem 5.2} with $p=2$. 
\end{proof}
\begin{example}
{\em For $m=3$, we have }
\begin{dmath*}
\sum_{1 \leq N_1<N_2<N_3}{\frac{1}{N_3^4 N_2^4 N_1^4}}
=\frac{1}{6}\left(\sum_{N=1}^{n}{\frac{1}{N^4}}\right)^{3}
-\frac{1}{2}\left(\sum_{N=1}^{n}{\frac{1}{N^4}}\right)\left(\sum_{N=1}^{n}{\frac{1}{N^8}}\right)
+\frac{1}{3}\left(\sum_{N=1}^{n}{\frac{1}{N^{12}}}\right)
={\frac{\pi^{12}}{681080400} 
=\frac{2(2^{2(3)})\pi^{4(3)}}{(4(3)+2)!} 
(\approx 0.001357063251)}
.\end{dmath*}
\end{example}
\begin{theorem} \label{Conjecture 5.2}
For any $m \in \mathbb{N}$, we have that 
$$
\sum_{1 \leq N_1 < \cdots < N_m}{\frac{1}{N_m^6 \cdots N_1^6}}
=\sum_{\sum{i.y_{k,i}}=m}{\prod_{i=1}^{m}{\frac{(-1)^{y_{k,i}}}{(y_{k,i})!i^{y_{k,i}}} \left(\zeta(6i)\right)^{y_{k,i}}}}
=\frac{6(2\pi)^{6m}}{(6m+3)!}
$$
and 
$$
\sum_{\sum{i.y_{k,i}}=m}{\prod_{i=1}^{m}{\frac{1}{(y_{k,i})!} \left(\frac{B_{6i}}{(2i)(6i)!}\right)^{y_{k,i}}}}
=\frac{6}{(6m+3)!}
.$$
\end{theorem}
\begin{proof}
From \cite{borwein1996evaluations}, we have that 
$$
\sum_{1 \leq N_1 < \cdots < N_m}{\frac{1}{N_m^6 \cdots N_1^6}}
=\frac{6(2\pi)^{6m}}{(6m+3)!}
.$$
Hence, applying Theorem \ref{Theorem 3.1}, we obtain the first equation of this theorem. \\
The second equation is obtained by applying Theorem \ref{Theorem 5.2} with $p=3$. 
\end{proof}
\begin{example}
{\em For $m=3$, we have }
\begin{dmath*}
\sum_{1 \leq N_1<N_2<N_3}{\frac{1}{N_3^6 N_2^6 N_1^6}}
=\frac{1}{6}\left(\sum_{N=1}^{n}{\frac{1}{N^6}}\right)^{3}
-\frac{1}{2}\left(\sum_{N=1}^{n}{\frac{1}{N^6}}\right)\left(\sum_{N=1}^{n}{\frac{1}{N^{12}}}\right)
+\frac{1}{3}\left(\sum_{N=1}^{n}{\frac{1}{N^{18}}}\right)
={\frac{2\pi^{18}}{64965492466875} 
=\frac{6(2\pi)^{6(3)}}{(6(3)+3)!} 
(\approx 0.00002735551966)}
.\end{dmath*}
\end{example}
\section{Relation to Recurrent Sums and Odd-Even Partition Identities} \label{Relation to Recurrent Sums and Odd-Even Partitions}
Recurrent sums and multiple sums have been studied separately respectively in \cite{RecurrentSums} and in this paper. In this section. we will compare these types of sums and show their similarities and the link between them. Then by combining the individual relations of each of these sums, we will produce new results. In particular, we will obtain new relations governing odd partitions and even partitions of a non-negative integer.  
\subsection{Relations between recurrent sums and multiple sums}
In this section, we develop the relation linking recurrent and multiple sums. Recurrent sums and multiple sums can be related by the following theorem.  
\begin{theorem} \label{Theorem 6.1}
For any $m,q,n \in \mathbb{N}$ and for any sequence $a_N$ defined in the interval $[q,n]$, we have that 
\begin{dmath*}
\sum_{N_m=q}^{n}{\cdots \sum_{N_1=q}^{N_2}{a_{N_m}\cdots a_{N_1}}}
+(-1)^m\sum_{q \leq N_1 < \cdots <N_m \leq n}{a_{N_m}\cdots a_{N_1}}
=2\sum_{\substack{\sum{i.y_{k,i}}=m \\ \sum{y_{k,i}} \,\, is \,\, even}}{\prod_{i=1}^{m}{\frac{1}{(y_{k,i})!} \left( \frac{1}{i}\sum_{N=q}^{n}{(a_N)^i }\right)^{y_{k,i}}}}
.\end{dmath*}
\begin{dmath*}
\sum_{N_m=q}^{n}{\cdots \sum_{N_1=q}^{N_2}{a_{N_m}\cdots a_{N_1}}}
-(-1)^m\sum_{q \leq N_1 < \cdots <N_m \leq n}{a_{N_m}\cdots a_{N_1}}
=2\sum_{\substack{\sum{i.y_{k,i}}=m \\ \sum{y_{k,i}} \,\, is \,\, odd}}{\prod_{i=1}^{m}{\frac{1}{(y_{k,i})!} \left( \frac{1}{i}\sum_{N=q}^{n}{(a_N)^i }\right)^{y_{k,i}}}}
.\end{dmath*}
\end{theorem}
\begin{proof}
We can notice that 
\begin{dmath*}
\sum_{\sum{i.y_{k,i}}=m}{\prod_{i=1}^{m}{\frac{1}{(y_{k,i})!} \left( \frac{1}{i}\sum_{N=q}^{n}{(a_N)^i }\right)^{y_{k,i}}}}
+\sum_{\sum{i.y_{k,i}}=m}{\prod_{i=1}^{m}{\frac{(-1)^{y_{k,i}}}{(y_{k,i})!} \left( \frac{1}{i}\sum_{N=q}^{n}{(a_N)^i }\right)^{y_{k,i}}}}
=2\sum_{\substack{\sum{i.y_{k,i}}=m \\ \sum{y_{k,i}} \,\, is \,\, even}}{\prod_{i=1}^{m}{\frac{1}{(y_{k,i})!} \left( \frac{1}{i}\sum_{N=q}^{n}{(a_N)^i }\right)^{y_{k,i}}}}
.\end{dmath*}
\begin{dmath*}
\sum_{\sum{i.y_{k,i}}=m}{\prod_{i=1}^{m}{\frac{1}{(y_{k,i})!} \left( \frac{1}{i}\sum_{N=q}^{n}{(a_N)^i }\right)^{y_{k,i}}}}
-\sum_{\sum{i.y_{k,i}}=m}{\prod_{i=1}^{m}{\frac{(-1)^{y_{k,i}}}{(y_{k,i})!} \left( \frac{1}{i}\sum_{N=q}^{n}{(a_N)^i }\right)^{y_{k,i}}}}
=2\sum_{\substack{\sum{i.y_{k,i}}=m \\ \sum{y_{k,i}} \,\, is \,\, odd}}{\prod_{i=1}^{m}{\frac{1}{(y_{k,i})!} \left( \frac{1}{i}\sum_{N=q}^{n}{(a_N)^i }\right)^{y_{k,i}}}}
.\end{dmath*}
From \cite{RecurrentSums}, we have 
$$\sum_{N_m=q}^{n}{\cdots \sum_{N_1=q}^{N_2}{a_{N_m}\cdots a_{N_1}}}
=\sum_{\sum{i.y_{k,i}}=m}{\prod_{i=1}^{m}{\frac{1}{(y_{k,i})!} \left( \frac{1}{i}\sum_{N=q}^{n}{(a_N)^i }\right)^{y_{k,i}}}}
.$$
From Theorem \ref{Theorem 3.1}, we have 
$$\sum_{q \leq N_1 < \cdots <N_m \leq n}{a_{N_m}\cdots a_{N_1}}
=(-1)^m\sum_{\sum{i.y_{k,i}}=m}{\prod_{i=1}^{m}{\frac{(-1)^{y_{k,i}}}{(y_{k,i})!} \left( \frac{1}{i}\sum_{N=q}^{n}{(a_N)^i }\right)^{y_{k,i}}}}
.$$
Hence, by combining these relations, we obtain the theorem. 
\end{proof}
\begin{example}
{\em For $m=2$, Theorem \ref{Theorem 6.1} gives }
$$
\sum_{N_2=q}^{n}{\sum_{N_1=q}^{N_2}{a_{N_2} a_{N_1}}}
+\sum_{q \leq N_1 <N_2 \leq n}{a_{N_2} a_{N_1}}
=\left(\sum_{N=1}^{n}{a_N}\right)^{2}
.$$
$$
\sum_{N_2=q}^{n}{\sum_{N_1=q}^{N_2}{a_{N_2} a_{N_1}}}
-\sum_{q \leq N_1 <N_2 \leq n}{a_{N_2} a_{N_1}}
=\left(\sum_{N=1}^{n}{{(a_N)}^{2}}\right)
.$$
\end{example}
\begin{example}
{\em For $m=3$, Theorem \ref{Theorem 6.1} gives }
$$
\sum_{N_3=1}^{n}{\sum_{N_2=1}^{N_3}{\sum_{N_1=1}^{N_2}{a_{N_3}a_{N_2}a_{N_1}}}}
+\sum_{1 \leq N_1 < N_2 < N_3 \leq n}{a_{N_3}a_{N_2}a_{N_1}}
=\frac{1}{3}\left(\sum_{N=1}^{n}{a_{N}}\right)^{3}
+\frac{2}{3}\left(\sum_{N=1}^{n}{\left(a_{N}\right)^{3}}\right)
.$$
$$
\sum_{N_3=1}^{n}{\sum_{N_2=1}^{N_3}{\sum_{N_1=1}^{N_2}{a_{N_3}a_{N_2}a_{N_1}}}}
-\sum_{1 \leq N_1 < N_2 < N_3 \leq n}{a_{N_3}a_{N_2}a_{N_1}}
=\left(\sum_{N=1}^{n}{a_{N}}\right)\left(\sum_{N=1}^{n}{\left(a_{N}\right)^{2}}\right)
.$$
\end{example}
\begin{example}
{\em For $m=4$, Theorem \ref{Theorem 6.1} gives }
\begin{dmath*}
\sum_{N_4=1}^{n}{\sum_{N_3=1}^{N_4}{\sum_{N_2=1}^{N_3}{\sum_{N_1=1}^{N_2}{a_{N_4}a_{N_3}a_{N_2}a_{N_1}}}}}
+\sum_{1 \leq N_1 < N_2 < N_3 < N_4 \leq n}{a_{N_4}a_{N_3}a_{N_2}a_{N_1}}
=\frac{1}{12}\left(\sum_{N=1}^{n}{a_{N}}\right)^{4}
+\frac{2}{3}\left(\sum_{N=1}^{n}{a_{N}}\right)\left(\sum_{N=1}^{n}{\left(a_{N}\right)^{3}}\right)
+\frac{1}{4}\left(\sum_{N=1}^{n}{\left(a_{N}\right)^{2}}\right)^{2}
.\end{dmath*}
\begin{dmath*}
\sum_{N_4=1}^{n}{\sum_{N_3=1}^{N_4}{\sum_{N_2=1}^{N_3}{\sum_{N_1=1}^{N_2}{a_{N_4}a_{N_3}a_{N_2}a_{N_1}}}}}
-\sum_{1 \leq N_1 < N_2 < N_3 < N_4 \leq n}{a_{N_4}a_{N_3}a_{N_2}a_{N_1}}
=
\frac{1}{2}\left(\sum_{N=1}^{n}{a_{N}}\right)^{2}\left(\sum_{N=1}^{n}{\left(a_{N}\right)^{2}}\right)
+\frac{1}{2}\left(\sum_{N=1}^{n}{\left(a_{N}\right)^{4}}\right)
.\end{dmath*}
\end{example}
\subsection{Odd and Even Partition Identities}
In \cite{RecurrentSums} and in this paper, we have produced multiple partition identities. Combining these identities, we are able to produce several identities for even and odd partitions. Note that a partition $(y_{k,1}, \cdots, y_{k,m})$ is odd if $\sum{y_{k,i}}$ is odd and even if $\sum{y_{k,i}}$ is even. 
\begin{theorem} \label{Theorem 6.2}
Let $m$ be a non-negative integer, 
$$
\sum_{\substack{\sum{i.y_{k,i}}=m \\ \sum{y_{k,i}} \,\, is \,\, even}}{\prod_{i=1}^{m}{\frac{1}{i^{y_{k,i}}(y_{k,i})!}}}
=
\begin{cases}
1 & ${\em for} $ m=0, \\
0 & ${\em for} $ m=1, \\
\frac{1}{2} & ${\em for} $ m \geq 2. \\
\end{cases}
$$
$$
\sum_{\substack{\sum{i.y_{k,i}}=m \\ \sum{y_{k,i}} \,\, is \,\, odd}}{\prod_{i=1}^{m}{\frac{1}{i^{y_{k,i}}(y_{k,i})!}}}
=
\begin{cases}
0 & ${\em for} $ m=0, \\
1 & ${\em for} $ m=1, \\
\frac{1}{2} & ${\em for} $ m \geq 2. \\
\end{cases}
$$
\end{theorem}
\begin{proof}
We can notice that 
$$
\sum_{\substack{k \\ \sum{i.y_{k,i}}=m}}{\prod_{i=1}^{m}{\frac{1}{i^{y_{k,i}}(y_{k,i})!}}}
+\sum_{\substack{k \\ \sum{i.y_{k,i}}=m}}{\prod_{i=1}^{m}{\frac{(-1)^{y_{k,i}}}{i^{y_{k,i}}(y_{k,i})!}}}
=2\sum_{\substack{\sum{i.y_{k,i}}=m \\ \sum{y_{k,i}} \,\, is \,\, even}}{\prod_{i=1}^{m}{\frac{1}{i^{y_{k,i}}(y_{k,i})!}}}
.$$
$$
\sum_{\substack{k \\ \sum{i.y_{k,i}}=m}}{\prod_{i=1}^{m}{\frac{1}{i^{y_{k,i}}(y_{k,i})!}}}
-\sum_{\substack{k \\ \sum{i.y_{k,i}}=m}}{\prod_{i=1}^{m}{\frac{(-1)^{y_{k,i}}}{i^{y_{k,i}}(y_{k,i})!}}}
=2\sum_{\substack{\sum{i.y_{k,i}}=m \\ \sum{y_{k,i}} \,\, is \,\, odd}}{\prod_{i=1}^{m}{\frac{1}{i^{y_{k,i}}(y_{k,i})!}}}
.$$
From \cite{RecurrentSums}, we have 
$$
\sum_{\substack{k \\ \sum{i.y_{k,i}}=m}}{\prod_{i=1}^{m}{\frac{1}{i^{y_{k,i}}(y_{k,i})!}}}
=1
.$$
From Lemma \ref{Lemma 3.1}, we have 
$$
\sum_{\substack{k \\ \sum{i.y_{k,i}}=m}}{\prod_{i=1}^{m}{\frac{(-1)^{y_{k,i}}}{i^{y_{k,i}}(y_{k,i})!}}}
=
\begin{cases}
(-1)^m & $for $ 0 \leq m \leq 1, \\
0 & $for $ m \geq 2. \\
\end{cases}
$$
Hence, by combining these relations, we obtain the theorem. 
\end{proof}
\begin{theorem} \label{Theorem 6.3}
Let $(y_{k,1}, \cdots , y_{k,m})=\{(y_{1,1}, \cdots , y_{1,m}),(y_{2,1}, \cdots , y_{2,m}), \cdots \}$ be the set of all partitions of $m$. Let $(\varphi_1, \cdots , \varphi_m)$ be a partition of $r \leq m$. We have that 
\begin{equation*}
\begin{split}
\sum_{\substack{\sum{i.y_{k,i}}=m \\ \sum{y_{k,i}} \,\, is \,\, even}}{\prod_{i=1}^{m}{\frac{\binom{y_{k,i}}{\varphi_i}}{i^{y_{k,i}}(y_{k,i})!}}}
&=\sum_{\substack{\sum{i.y_{k,i}}=m \\ \sum{y_{k,i}} \,\, is \,\, even \\ y_{k,i} \geq \varphi_i}}{\prod_{i=1}^{m}{\frac{\binom{y_{k,i}}{\varphi_i}}{i^{y_{k,i}}(y_{k,i})!}}} \\
&=
\begin{cases}
\frac{1+(-1)^{\sum{\varphi_i}}}{2}\prod_{i=1}^{m}{\frac{1}{i^{\varphi_{i}} (\varphi_{i})!}} & ${\em for} $ m-r=0, \\
\frac{1-(-1)^{\sum{\varphi_i}}}{2}\prod_{i=1}^{m}{\frac{1}{i^{\varphi_{i}} (\varphi_{i})!}} & ${\em for} $ m-r=1, \\
\frac{1}{2}\prod_{i=1}^{m}{\frac{1}{i^{\varphi_{i}} (\varphi_{i})!}} & ${\em for} $ m-r \geq 2. \\
\end{cases}
\end{split}
\end{equation*}
\begin{equation*}
\begin{split}
\sum_{\substack{\sum{i.y_{k,i}}=m \\ \sum{y_{k,i}} \,\, is \,\, odd}}{\prod_{i=1}^{m}{\frac{\binom{y_{k,i}}{\varphi_i}}{i^{y_{k,i}}(y_{k,i})!}}}
&=\sum_{\substack{\sum{i.y_{k,i}}=m \\ \sum{y_{k,i}} \,\, is \,\, odd \\ y_{k,i} \geq \varphi_i}}{\prod_{i=1}^{m}{\frac{\binom{y_{k,i}}{\varphi_i}}{i^{y_{k,i}}(y_{k,i})!}}} \\
&=
\begin{cases}
\frac{1-(-1)^{\sum{\varphi_i}}}{2}\prod_{i=1}^{m}{\frac{1}{i^{\varphi_{i}} (\varphi_{i})!}} & ${\em for} $ m-r=0, \\
\frac{1+(-1)^{\sum{\varphi_i}}}{2}\prod_{i=1}^{m}{\frac{1}{i^{\varphi_{i}} (\varphi_{i})!}} & ${\em for} $ m-r=1, \\
\frac{1}{2}\prod_{i=1}^{m}{\frac{1}{i^{\varphi_{i}} (\varphi_{i})!}} & ${\em for} $ m-r \geq 2. \\
\end{cases}
\end{split}
\end{equation*}
\end{theorem}
\begin{proof}
We can notice that 
$$
\sum_{\substack{k \\ \sum{i.y_{k,i}}=m}}{\prod_{i=1}^{m}{\frac{\binom{y_{k,i}}{\varphi_i}}{i^{y_{k,i}}(y_{k,i})!}}}
+\sum_{\substack{k \\ \sum{i.y_{k,i}}=m}}{\prod_{i=1}^{m}{\frac{(-1)^{y_{k,i}}\binom{y_{k,i}}{\varphi_i}}{i^{y_{k,i}}(y_{k,i})!}}}
=2\sum_{\substack{\sum{i.y_{k,i}}=m \\ \sum{y_{k,i}} \,\, is \,\, even}}{\prod_{i=1}^{m}{\frac{\binom{y_{k,i}}{\varphi_i}}{i^{y_{k,i}}(y_{k,i})!}}}
.$$
$$
\sum_{\substack{k \\ \sum{i.y_{k,i}}=m}}{\prod_{i=1}^{m}{\frac{\binom{y_{k,i}}{\varphi_i}}{i^{y_{k,i}}(y_{k,i})!}}}
-\sum_{\substack{k \\ \sum{i.y_{k,i}}=m}}{\prod_{i=1}^{m}{\frac{(-1)^{y_{k,i}}\binom{y_{k,i}}{\varphi_i}}{i^{y_{k,i}}(y_{k,i})!}}}
=2\sum_{\substack{\sum{i.y_{k,i}}=m \\ \sum{y_{k,i}} \,\, is \,\, odd}}{\prod_{i=1}^{m}{\frac{\binom{y_{k,i}}{\varphi_i}}{i^{y_{k,i}}(y_{k,i})!}}}
.$$
From \cite{RecurrentSums}, we have 
$$
\sum_{\substack{k \\ \sum{i.y_{k,i}}=m}}{\prod_{i=1}^{m}{\frac{\binom{y_{k,i}}{\varphi_i}}{i^{y_{k,i}}(y_{k,i})!}}}
=\sum_{\substack{k \\ \sum{i.y_{k,i}}=m \\ y_{k,i} \geq \varphi_i}}{\prod_{i=1}^{m}{\frac{\binom{y_{k,i}}{\varphi_i}}{i^{y_{k,i}}(y_{k,i})!}}}
=\prod_{i=1}^{m}{\frac{1}{i^{\varphi_{i}} (\varphi_{i})!}}
.$$
From Lemma \ref{Lemma 3.2}, we have 
\begin{dmath*}
\sum_{\substack{k \\ \sum{i.y_{k,i}}=m}}{\prod_{i=1}^{m}{\frac{(-1)^{y_{k,i}}\binom{y_{k,i}}{\varphi_i}}{i^{y_{k,i}}(y_{k,i})!}}}
=\sum_{\substack{k \\ \sum{i.y_{k,i}}=m \\ y_{k,i} \geq \varphi_i}}{\prod_{i=1}^{m}{\frac{(-1)^{y_{k,i}}\binom{y_{k,i}}{\varphi_i}}{i^{y_{k,i}}(y_{k,i})!}}}
=
\begin{cases}
(-1)^{m-r} \prod_{i=1}^{m}{\frac{(-1)^{\varphi_i}}{i^{\varphi_{i}} (\varphi_{i})!}} & $for $ 0 \leq m-r \leq 1, \\
0 & $for $ m-r \geq 2. \\
\end{cases}
\end{dmath*}
Hence, by combining these relations, we obtain the theorem. 
\end{proof}
\begin{theorem} \label{Theorem 6.4}
Let $m,n \in \mathbb{N}$, 
\begin{equation*}
\sum_{\substack{\sum{i.y_{k,i}}=m \\ \sum{y_{k,i}} \,\, is \,\, even}}{\prod_{i=1}^{m}{\frac{1}{(y_{k,i})!} \left( \frac{n}{i}\right)^{y_{k,i}}}}
=\frac{1}{2}\left[\binom{n-m+1}{m}+(-1)^m\binom{n}{m}\right]
.\end{equation*}
\begin{equation*}
\sum_{\substack{\sum{i.y_{k,i}}=m \\ \sum{y_{k,i}} \,\, is \,\, odd}}{\prod_{i=1}^{m}{\frac{1}{(y_{k,i})!} \left( \frac{n}{i}\right)^{y_{k,i}}}}
=\frac{1}{2}\left[\binom{n-m+1}{m}-(-1)^m\binom{n}{m}\right]
.\end{equation*}
\end{theorem}
\begin{proof}
We can notice that 
\begin{dmath*}
\sum_{\sum{i.y_{k,i}}=m}{\prod_{i=1}^{m}{\frac{1}{(y_{k,i})!} \left( \frac{n}{i}\right)^{y_{k,i}}}}
+\sum_{\sum{i.y_{k,i}}=m}{\prod_{i=1}^{m}{\frac{(-1)^{y_{k,i}}}{(y_{k,i})!} \left( \frac{n}{i}\right)^{y_{k,i}}}}
=2\sum_{\substack{\sum{i.y_{k,i}}=m \\ 
\sum{y_{k,i}} \,\, is \,\, even}}{\prod_{i=1}^{m}{\frac{1}{(y_{k,i})!} \left( \frac{n}{i}\right)^{y_{k,i}}}}
.\end{dmath*}
\begin{dmath*}
\sum_{\sum{i.y_{k,i}}=m}{\prod_{i=1}^{m}{\frac{1}{(y_{k,i})!} \left( \frac{n}{i}\right)^{y_{k,i}}}}
-\sum_{\sum{i.y_{k,i}}=m}{\prod_{i=1}^{m}{\frac{(-1)^{y_{k,i}}}{(y_{k,i})!} \left( \frac{n}{i}\right)^{y_{k,i}}}}
=2\sum_{\substack{\sum{i.y_{k,i}}=m \\ 
\sum{y_{k,i}} \,\, is \,\, odd}}{\prod_{i=1}^{m}{\frac{1}{(y_{k,i})!} \left( \frac{n}{i}\right)^{y_{k,i}}}}
.\end{dmath*}
From \cite{RecurrentSums}, we have 
$$\sum_{\sum{i.y_{k,i}}=m}{\prod_{i=1}^{m}{\frac{1}{(y_{k,i})!} \left( \frac{n}{i}\right)^{y_{k,i}}}}=\binom{n+m-1}{m}
.$$
From Corollary \ref{Corollary 3.2}, we have 
$$\sum_{\sum{i.y_{k,i}}=m}{\prod_{i=1}^{m}{\frac{(-1)^{y_{k,i}}}{(y_{k,i})!} \left( \frac{n}{i}\right)^{y_{k,i}}}}=(-1)^m\binom{n}{m}
.$$
Hence, by combining these relations, we obtain the theorem. 
\end{proof}
\bibliographystyle{apa}
\bibliography{Multiple_Sums_Arxiv}

\end{document}